\documentclass[10pt,a4paper,twoside]{amsart}

\usepackage{graphicx}
\usepackage{a4wide}

\newtheorem{definition}{Definition}[section]
\newtheorem{theorem}{Theorem}[section]
\newtheorem{lemma}{Lemma}[section]

\newtheorem*{maintheorem*}{Main Theorem}
\allowdisplaybreaks
\numberwithin{equation}{section}

\newcommand{\norm}[1]{\left\| #1 \right\|}

\newcommand{\eps}{\varepsilon}

\newcommand{\eb}{{\eps,\beta}}
\newcommand{\ebdg}{\eps,\beta,\delta,\gamma}
\newcommand{\uebdg}{u_{\ebdg}}
\newcommand{\Pebdg}{P_{\ebdg}}
\newcommand{\ueb}{u_\eb}

\newcommand{\Peb}{P_{\eps,\beta}}
\newcommand{\uedg}{u_{\eps,\delta,\gamma}}
\newcommand{\Pedg}{P_{\eps,\delta,\gamma}}
\newcommand{\Pe}{P_\eps}

\newcommand{\gk}{\gamma_{k}}
\newcommand{\dk}{\delta_k}
\newcommand{\uedgk}{u_{\eps_{k},\dk,\gk}}
\newcommand{\Pebdgk}{P_{\eps_{k},\beta_{k},\dk,\gk}}

\newcommand{\epsk}{\eps_{k}}

\newcommand{\pt}{\partial_t}

\newcommand{\px}{\partial_x }
\newcommand{\pxx}{\partial_{xx}^2}
\newcommand{\pxxx}{\partial_{xxx}^3}
\newcommand{\pxxxx}{\partial_{xxxx}^4}

\newcommand{\ptx}{\partial_{tx}^2}

\renewcommand{\i}{\ifmmode\mathit{\mathchar"7010 }\else\char"10 \fi}
\renewcommand{\j}{\ifmmode\mathit{\mathchar"7011 }\else\char"11 \fi}
\newcommand{\R}{\mathbb{R}}
\newcommand{\N}{\mathbb{N}}

\newcommand{\supp}{\mathrm{supp}\,}
\newcommand{\Hneg}{H_{\mathrm{loc}}^{-1}}
\newcommand{\CL}{\mathcal{L}}

\newcommand{\Pedgk}{P_{\epsk,\,\dk,\,\gk}}

{%

\begin{enumerate}}%
{\end{enumerate}}

{%

\begin{enumerate}}%
{\end{enumerate}}

\begin{document}\large

\title[Singular limits]{Dispersive and Diffusive limits\\ for Ostrovsky-Hunter type equations}

\date{\today}

\author[G. M. Coclite and L. di Ruvo]{Giuseppe Maria Coclite and Lorenzo di Ruvo}
\address[Giuseppe Maria Coclite and Lorenzo di Ruvo]
{\newline Department of Mathematics,   University of Bari, via E. Orabona 4, 70125 Bari,   Italy}
\email[]{giuseppemaria.coclite@uniba.it, lorenzo.diruvo@uniba.it}
\urladdr{http://www.dm.uniba.it/Members/coclitegm/}

\thanks{The authors are members of the Gruppo Nazionale per l'Analisi Matematica, la Probabilit\`a e le loro Applicazioni (GNAMPA) of the Istituto Nazionale di Alta Matematica (INdAM)}




\begin{abstract}
We consider the  equation
\begin{equation*}
\px(\pt u+\px f(u)-\beta \pxxx u)=\gamma u,
\end{equation*}
that includes the short pulse, the Ostrovsky-Hunter, and the Korteweg-deVries ones. We consider here the asymptotic behavior as $\gamma\to 0$.
The proof relies on deriving suitable a priori estimates together with an application of the compensated compactness method in the $L^p$ setting.
\end{abstract}

\maketitle


\section{Introduction}
\label{sec:intro}
The nonlinear evolution equation
\begin{equation}
\label{eq:Kdv}
\px(\pt u+\px f(u)-\beta \pxxx u)=0,
\end{equation}
with $\beta\in \R$  and $f(u)=\frac{u^2}{2}$, was derived by Korteweg-deVries to model internal solitary waves in the atmosphere and ocean.
Here $u(t,x)$ is the amplitude of an appropriate linear long wave mode, with linear long wave speed $C_0$. However, when the effects of background
rotation through the Coriolis parameter $\kappa$ need to be taken into account, an extra term is needed, and \eqref{eq:Kdv} is replaced by
\begin{equation}
\label{eq:OHbeta}
\px(\pt u+\px f(u)-\beta \pxxx u)=\gamma u,
\end{equation}
where $\gamma=\frac{\kappa^2}{2C_0}$ (see \cite{dR,HT}), which is known as the Ostrovsky equation (see \cite{O}).

Mathematical properties of the Ostrovsky equation \eqref{eq:OHbeta} were studied
recently in many details, including the local and global well-posedness in energy space
\cite{GL, LM, LV, T}, stability of solitary waves \cite{LL, L, LV:JDE}, and convergence of solutions in the
limit of the Korteweg-deVries equation \cite{LL:07, LV:JDE}.
We shall consider the limit of no high-frequency dispersion $\beta=0$, therefore \eqref{eq:OHbeta} reads
\begin{equation}
\label{eq:OH}
\px(\pt u+\px f(u))=\gamma u, \quad f(u)=\frac{u^2}{2}.
\end{equation}
\eqref{eq:OH} is deduced considering two asymptotic expansions of the shallow water equations, first with respect to the rotation frequency and then with respect to the amplitude of the waves (see \cite{dR,HT}). It is known under different names such as the reduced Ostrovsky equation \cite{P, S}, the
Ostrovsky-Hunter equation \cite{B}, the short-wave equation \cite{H}, and the Vakhnenko equation
\cite{MPV, PV}.

Integrating \eqref{eq:OH} on $x$ we gain the integro-differential formulation of  \eqref{eq:OH} (see \cite{LPS})
\begin{equation}
\label{eq:700}
\pt u+\px f(u)=\gamma \int^x u(t,y) dy,
\end{equation}
that is equivalent to
\begin{equation}
\label{eq:800}
\pt u+\px f(u)=\gamma P,\qquad \px P=u.
\end{equation}

The unique useful conserved quantities are 
\begin{equation}
\label{eq:conservedq}
t\longmapsto\int u(t,x)dx=0,\qquad t\longmapsto\int u^2(t,x)dx.
\end{equation}
In the sense that if $u(t,\cdot)$ has zero mean at time $t=0$, then it will have zero mean at any time $t>0$.
In addition, the $L^2$ norm  of $u(t,\cdot)$ is constant with respect to $t$.

In \cite{Cd3, CdK, dR}, it is proved that \eqref{eq:OH} admits an unique entropy solutions in the sense of the following definition
\begin{definition}
\label{def:sol}
We say that $u\in  L^{\infty}((0,T)\times\R)$, $T>0$, is an entropy solution of the initial
value problem \eqref{eq:OH}, if
\begin{itemize}
\item[$i$)] $u$ is a distributional solution of \eqref{eq:700} or equivalently of \eqref{eq:800};
\item[$ii$)] for every convex function $\eta\in  C^2(\R)$ the
entropy inequality
\begin{equation}
\label{eq:OHentropy}
\pt \eta(u)+ \px q(u)-\gamma\eta'(u) P\le 0, \qquad     q(u)=\int^u f'(\xi) \eta'(\xi)\, d\xi,
\end{equation}
holds in the sense of distributions in $(0,\infty)\times\R$.
\end{itemize}
\end{definition}
In \cite{Cd}, it is proved the wellposedness of the entropy solutions of \eqref{eq:700}, or  \eqref{eq:800}, for the non-homogeneous initial boundary problem, while in \cite{Cd2} it is proved the convergence of the solutions of \eqref{eq:OHbeta} to the  discontinuous solutions of \eqref{eq:700}, or \eqref{eq:800}.

If $f(u)=-\frac{1}{6}u^3$, \eqref{eq:OHbeta} reads,
\begin{equation}
\label{eq:RSHEP}
\px\left(\pt u-\frac{1}{6}\px(u^3)  -\beta \pxxx u\right)=\gamma u.
\end{equation}
\eqref{eq:RSHEP} is known as the regularized short pulse equation, and was derived by Costanzino, Manukian and Jones \cite{CMJ}
in the context of the nonlinear Maxwell equations with high-frequency dispersion.

If we send $\beta\to 0$ in \eqref{eq:RSHEP}, we pass from \eqref{eq:RSHEP} to the equation
\begin{equation}
\label{eq:SHP}
\px\left(\pt u-\frac{1}{6}\px (u^3)\right)=\gamma u,
\end{equation}
or equivalently (see \cite{SS}),
\begin{equation}
\label{eq:SHP-bis}
\pt u-\frac{1}{6}\px (u^3)=\gamma P, \quad \px P=u.
\end{equation}
\eqref{eq:SHP} is known as the short pulse equation, and was introduced recently by
Sch\"afer and Wayne \cite{SW} as  a model equation describing the propagation of ultra-short light pulses in silica optical fibers.
It provides also an approximation of nonlinear wave packets in dispersive media in the limit of few cycles on the ultra-short pulse scale.
In \cite{Cd1,CdK, dR}, it is proved the wellposedness of the entropy solution of \eqref{eq:SHP} in sense of Definition \eqref{def:sol}, for the initial boundary problem and for the Cauchy problem, while, in \cite{Cd4}, it is proved the convergence of the solutions of \eqref{eq:RSHEP} to the  discontinuous solutions of \eqref{eq:SHP}.

The deep difference between the two equations is in the flux.
If we have a function that preserves the conserved quantities we can make sense of \eqref{eq:OH} using the 
distribution theory because the flux is quadratic and the $L^2$ norm is preserved. On the contrary the same argument does not apply 
to \eqref{eq:SHP}. Indeed, the flux is cubic and we do not have any information on the $L^3$ norm of the solution.
In \cite{Cd1}, we solved this problem proving that the solutions are bounded, and the argument is much more delicate than the one in \cite{Cd}.

In this paper, we study the dispersion-diffusion of \eqref{eq:OHbeta} and of \eqref{eq:800}, when $\gamma\to 0$ 
(that is, when $\kappa\to 0$, or $C_0 \to\infty$). We prove that, if $\gamma \to 0$, the solution of \eqref{eq:OHbeta} and of \eqref{eq:800} converge to the to the  discontinuous solutions of the following equation
\begin{equation}
\label{eq:Bur}
\pt u +\px(u^2)=0,
\end{equation}
which is known as Burgers' equation.
Likewise, when $\gamma\to 0$, the solutions of \eqref{eq:RSHEP} and of \eqref{eq:SHP} converge to the discontinuous solutions of the following  scalar conservation law
\begin{equation}
\label{eq:CL}
\pt u -\frac{1}{6}\px (u^3)=0.
\end{equation}

The paper is organized in three sections. In Section \ref{sec:OH} we prove the convergence of \eqref{eq:800} and of \eqref{eq:SHP} to \eqref{eq:Bur} and \eqref{eq:CL}, respectively. 
In Section \ref{sec:OE}, we prove the convergence of \eqref{eq:OHbeta} to the \eqref{eq:Bur}, while in Section \ref{sec:RSPE}, we prove the convergence of \eqref{eq:RSHEP} to \eqref{eq:CL}.

\section{Ostrovsky-Hunter equation and short pulse one: $\gamma\to 0$.}\label{sec:OH}
In this section, we consider the following Cauchy problem
\begin{equation}
\label{eq:cl-p}
\begin{cases}
\pt u +\px f(u)=\gamma P, &\quad t>0,\, x\in\R,\\
\px P =u , &\quad t>0,\, x\in\R,\\
P(t,-\infty)=0, &\quad t>0,\\
u(0,x)=u_0(x), &\quad x\in\R,
\end{cases}
\end{equation}
or equivalently,
\begin{equation}
\label{eq:cl-p1}
\begin{cases}
\pt u +\px f(u)=\gamma \int_{-\infty}^{x} u(t,y) dy, &\quad t>0,\, x\in\R,\\
u(0,x)=u_0(x), &\quad x\in\R.
\end{cases}
\end{equation}
On the initial datum, we assume that
\begin{equation}
\label{eq:assinit}
u_0\in L^1(\R)\cap L^{\infty}(\R),\quad\int_{\R}u_{0}(x)dx=0,
\end{equation}
while, on the function
\begin{equation}
\label{eq:def-di-P0}
P_{0}(x)=\int_{-\infty}^{x} u_{0}(y)dy, \quad x\in\R,
\end{equation}
we assume that
\begin{equation}
\label{eq:L-2P0}
\int_{\R} P_0(x) dx=\int_{\R}\left(\int_{-\infty}^{x}u_{0}(y)dy\right)dx =0.
\end{equation}
Moreover, the flux $f\in C^2(\R)$ is assumed to be smooth.\\
If $\gamma=0$, \eqref{eq:cl-p} reads
\begin{equation}
\label{eq:cl}
\begin{cases}
\pt u +\px f(u)=0,&\quad t>0,\, x\in\R,\\
u(0,x)=u_0(x), &\quad x\in\R,
\end{cases}
\end{equation}
which is a scalar conservation law.

Fix three small numbers $0<\eps,\delta,\gamma<1 $, and let $\uedg=\uedg(t,x)$ be the unique
classical solution of the following mixed problem:
\begin{equation}
\label{eq:OHepsw}
\begin{cases}
\pt \uedg+\px f(\uedg)=\gamma\Pedg+ \eps\pxx\uedg, &\quad t>0,\, x\in\R,\\
-\delta\pt\Pedg +\px\Pedg=\uedg,&\quad t>0, x\in\R,\\
\Pedg(t,-\infty)=0, &\quad t>0,\\
\uedg(0,x)=u_{\eps, \delta, \gamma, 0}(x),&\quad x\in\R,
\end{cases}
\end{equation}
where $u_{\eps,\delta, \gamma,0}$ is a $C^\infty$ approximation of $u_{0}$ such that
\begin{equation}
\label{eq:u0eps}
\begin{split}
&u_{\,\eps,\,\delta,\,\gamma,0} \to u_{0} \quad  \textrm{in $L^{p}_{loc}(\R)$, $1\le p <\infty$, as $\eps,\,\delta,\,\gamma \to 0$,}\\
&\norm{u_{\eps,\delta,\gamma,0}}_{L^{\infty}(\R)}\le \norm{u_0}_{L^{\infty}(\R)}, \quad \eps,\,\delta,\,\gamma>0,\\
&\norm{u_{\eps,\delta,\gamma,0}}^2_{L^2(\R)}+\delta\gamma\norm{P_{\eps,\delta,\gamma,0}}^2_{L^2(\R)}+ \delta\norm{\px P_{\eps,\delta,\gamma,0}}^2_{L^2(\R)}\le C_{0}, \quad \eps,\,\delta,\,\gamma>0,\\
&\int_{\R}u_{\eps,\delta,\gamma,0}(x) dx =0,\quad \int_{\R}P_{\eps,\delta,\gamma,0}(x) dx =0, \quad \eps,\,\delta,\,\gamma>0,
\end{split}
\end{equation}
and $C_0$ is a constant independent on $\eps$, $\delta$ and $\gamma$.

The main result of this section is the following theorem.
\begin{theorem}
\label{th:main-1}
Fix $T>0$. Assume \eqref{eq:assinit}, \eqref{eq:def-di-P0}, \eqref{eq:L-2P0} and \eqref{eq:u0eps}. If
\begin{equation}
\label{eq:def-di-gamma}
\gamma=\mathbf{\mathcal{O}}( \eps^{\frac{1}{3}}\delta).
\end{equation}
There exists three sequences $\{\eps_{k}\}_{k\in\N}$, $\{\beta_{k}\}_{k\in\N}$, $\{\gamma_{k}\}_{k\in\N}$, with $\eps_k, \beta_k, \gamma_k \to 0$, such that
\begin{equation*}
u_{\eps_k, \beta_k, \gamma_k}\to u\quad  \textrm{strongly in $L^{p}_{loc}((0,T)\times\R)$, for each $1\le p <\infty$},
\end{equation*}
where $u$ is the unique entropy solution of \eqref{eq:cl}.
Moreover, we have that
\begin{equation}
\label{eq:umedianulla}
\int_{\R}u(t,x)dx =0, \quad t>0.
\end{equation}
\end{theorem}
Let us prove some a priori estimates on $\uedg$ and $\Pedg$, denoting with $C_0$ the constants which depend on the initial datum, and $C(T)$ the constants which depend also on $T$.
\begin{lemma}
\label{lm:cns-1}
For each $t>0$,
\begin{equation}
\label{eq:P-pxP-intfy-1}
\Pedg(t,\infty)=\px\Pedg(t,-\infty)=\px\Pedg(t,\infty)=0.
\end{equation}
In particular, we have that
\begin{equation}
\label{eq:up-1}
\int_{\R}\uedg(t,x)dx = -\delta\frac{d}{dt}\int_{\R}\Pedg(t,x)dx, \quad t>0.
\end{equation}
\end{lemma}
\begin{proof}
Arguing as \cite[Lemma $3.1$]{Cd}, we have \eqref{eq:P-pxP-intfy-1}.

Let us show that \eqref{eq:up-1} holds. Integrating the second equation in \eqref{eq:OHepsw} on $(-\infty,x)$, we have
\begin{equation}
\label{eq:1}
\begin{split}
\Pedg(t,x)=&\int_{-\infty}^{x}\uedg(t,y)dy  +\delta\int_{-\infty}^{x}\pt\Pedg(t,y) dy\\
 =& \int_{-\infty}^{x}\uedg(t,y)dy +\delta\frac{d}{dt}\int_{-\infty}^{x}\Pedg(t,y) dy.
\end{split}
\end{equation}
Therefore, \eqref{eq:up-1} follows from \eqref{eq:P-pxP-intfy-1} and \eqref{eq:1}.
\end{proof}

\begin{lemma}\label{lm:u-media-nulla}
For each $t>0$,
\begin{equation}
\label{eq:Pmedianilla}
\int_{\R}\Pedg(t,x)dx=0.
\end{equation}
In particular, we have that
\begin{equation}
\label{u-media-nulla-1}
\int_{\R}\uedg(t,x)dx=0, \quad t>0.
\end{equation}
\end{lemma}
\begin{proof}
Let $t>0$. Integrating the first equation in \eqref{eq:OHepsw} on $\R$, we have
\begin{equation}
\label{eq:2}
\int_{\R}\pt\uedg(t,x)dx =\frac{d}{dt}\int_{\R}\uedg(t,x)dx = \gamma \int_{\R}\Pedg(t,x) dx.
\end{equation}
Differentiating \eqref{eq:up-1} with respect to $t$, we get
\begin{equation}
\label{eq:3}
\frac{d}{dt}\int_{\R}\uedg(t,x)dx = -\delta\frac{d^2}{d^2t}\int_{\R}\Pedg(t,x)dx.
\end{equation}
Therefore, \eqref{eq:2} and \eqref{eq:3} give
\begin{equation*}
\frac{d^2}{d^2t}\int_{\R}\Pedg(t,x)dx+\frac{\gamma}{\delta}\int_{\R}\Pedg(t,x) dx=0.
\end{equation*}
Then,
\begin{equation}
\label{eq:sol1}
\int_{\R}\Pedg(t,x) dx =C_{1}\cos\left(\sqrt{\frac{\gamma}{\delta}}t\right)+C_{2}\sin \left(\sqrt{\frac{\gamma}{\delta}}t\right),
\end{equation}
where $C_1,\, C_2$ are two constants.

It follows from \eqref{eq:u0eps} that
\begin{equation}
\label{eq:pt-int-P}
\frac{d}{dt}\int_{\R}P_{\eps,\delta,\gamma,0}(x) dx =0.
\end{equation}
Thanks to \eqref{eq:u0eps} and \eqref{eq:pt-int-P}, to compute $C_1,\,C_2$, we must solve the following system:
\begin{equation}
\label{eq:sist-1}
\begin{cases}
\displaystyle &C_{1}\cos\left(\sqrt{\frac{\gamma}{\delta}}t\right)+C_{2}\sin \left(\sqrt{\frac{\gamma}{\delta}}t\right)=0,\\
\displaystyle &-C_{1}\sqrt{\frac{\gamma}{\delta}}\sin \left(\sqrt{\frac{\gamma}{\delta}}t\right)+C_{2}\sqrt{\frac{\gamma}{\delta}}\cos\left(\sqrt{\frac{\gamma}{\delta}}t\right)=0.
\end{cases}
\end{equation}
\eqref{eq:sist-1} says that
\begin{equation}
\label{eq:val-cost}
C_1=C_2=0.
\end{equation}
Then, \eqref{eq:sol1} and \eqref{eq:val-cost} give \eqref{eq:Pmedianilla}.

Finally, let us show that \eqref{u-media-nulla-1} holds. Differentiating \eqref{eq:Pmedianilla} with respect to $t$, we get
\begin{equation}
\label{eq:pt-nullo}
\frac{d}{dt}\int_{\R}\Pedg(t,x)dx=0, \quad t>0.
\end{equation}
Therefore, \eqref{u-media-nulla-1} follows from \eqref{eq:up-1} and \eqref{eq:pt-nullo}.
\end{proof}

\begin{lemma}\label{lm:stima-l-2}
For each $t>0$,
\begin{equation}
\label{stima-l-2}
\norm{\uedg(t,\cdot)}^2_{L^2(\R)} + \delta\gamma\norm{\Pedg(t,\cdot)}^2_{L^2(\R)} +2\eps\int_{0}^{t}\norm{\px\uedg(s,\cdot)}^2_{L^2(\R)}dx\le C_{0}.
\end{equation}
Moreover, fixed $T>0$, there exists $C(T)>0$, independent on $\eps$, $\delta$ and $\gamma$, such that
\begin{equation}
\label{eq:P-l-2-1}
\norm{\Pedg}_{L^{2}((0,T)\times\R)}\le \frac{C(T)}{\delta^{\frac{1}{2}}\gamma^{\frac{1}{2}}}, \quad 0<t<T.
\end{equation}
\end{lemma}
\begin{proof}
Let $t>0$. Multiplying by $\Pedg$ the second equation in \eqref{eq:OHepsw}, we have
\begin{equation}
-\delta\Pedg\pt\Pedg + \Pedg\px\Pedg =\uedg\Pedg.
\label{eq:up-2}
\end{equation}
Due to \eqref{eq:P-pxP-intfy-1}, an integration on $\R$ gives
\begin{equation}
\label{eq:up-4}
\begin{split}
-\frac{d}{dt}\left(\frac{\delta}{2}\norm{\Pedg(t,\cdot)}^2_{L^2(\R)}\right)=&\int_{\R}\uedg\Pedg dx -\frac{1}{2}\int_{\R}\px\left(\Pedg^2\right)dx\\
=&\int_{\R}\uedg\Pedg dx.
\end{split}
\end{equation}
Multiplying by $\uedg$ the first equation in \eqref{eq:OHepsw}, an integration on $\R$ gives
\begin{align*}
\frac{1}{2}\frac{d}{dt}\norm{\uedg(t,\cdot)}^2_{L^2(\R)}= &\int_{\R}\uedg\pt\uedg dx \\
=& -\int_{\R}\uedg f'(\uedg)\px\uedg dx + \gamma\int_{\R}\uedg\Pedg dx\\
& +\eps\int_{\R}\uedg\pxx\uedg dx\\
=& \gamma\int_{\R}\uedg\Pedg dx -\eps\norm{\px\uedg(t,x)}^2_{L^{2}(\R)},
\end{align*}
that is
\begin{equation}
\label{eq:5}
\frac{d}{dt}\norm{\uedg(t,\cdot)}^2_{L^2(\R)}+2\eps\norm{\px\uedg(t,x)}^2_{L^{2}(\R)}=  2\gamma\int_{\R}\uedg\Pedg dx.
\end{equation}
It follows from \eqref{eq:up-4} and \eqref{eq:5} that
\begin{equation}
\label{eq:12}
\frac{d}{dt}\left(\norm{\uedg(t,\cdot)}^2_{L^2(\R)} + \delta\gamma\norm{\Pedg(t,\cdot)}^2_{L^2(\R)}\right) +2\eps\norm{\px\uedg(t,x)}^2_{L^{2}(\R)}=0.
\end{equation}
Integrating \eqref{eq:12} on $(0,t)$, from \eqref{eq:u0eps}, we have \eqref{stima-l-2}.

Finally, we prove \eqref{eq:P-l-2-1}. Let $T>0$. We begin by observing that, from \eqref{stima-l-2}, we have that
\begin{equation*}
\delta\gamma\int_{\R}\Pedg^2 dx \le C_{0}.
\end{equation*}
An integration on $(0,T)$ gives
\begin{equation*}
\delta\gamma\int_{0}^{T}\!\!\!\int_{\R} \Pedg^2 dtdx \le C_{0}T= C(T),
\end{equation*}
that is  \eqref{eq:P-l-2-1}.
\end{proof}
\begin{lemma}\label{lm:10}
Let $T>0$. There exists $C(T)>0$, independent on $\eps$, $\delta$ and $\gamma$, such that
\begin{equation}
\label{eq:Px-in-l2}
\norm{\px\Pedg(t,\cdot)}_{L^2(\R)}\le \frac{C(T)}{\delta\sqrt{\eps}},
\end{equation}
for every $0<t<T$. Moreover,
\begin{equation}
\label{eq:p-l-infty}
\norm{\Pedg}_{L^{\infty}((0,T)\times\R)}\le \frac{C(T)}{\delta^{\frac{3}{4}}\gamma^{\frac{1}{4}}\eps^{\frac{1}{4}}}.
\end{equation}
\end{lemma}
\begin{proof}
Let $0<t<T$. Differentiating the second equation in \eqref{eq:OHepsw} with respect to $x$, we have
\begin{equation}
\label{eq:15}
\delta\ptx\Pedg= -\px\uedg + \pxx\Pedg.
\end{equation}
Multiplying \eqref{eq:15} by $\px\Pedg$, an integration on $\R$ and \eqref{eq:P-pxP-intfy-1} give
\begin{equation}
\begin{split}
\frac{d}{dt}\left(\delta\norm{\px\Pedg(t,\cdot)}^2_{L^2(\R)}\right)=&-2\int_{\R}\px\uedg\px\Pedg dx +\int_{\R}\px\left(\px\Pedg\right)^2 dx \\
=&-2\int_{\R}\px\uedg\px\Pedg dx.
\end{split}
\end{equation}
Due to the Young inequality,
\begin{align*}
-2\int_{\R}\px\uedg\px\Pedg dx\le & 2\left\vert\int_{\R}\px\uedg\px\Pedg dx\right\vert\\
\le &2\int_{\R}\left\vert\frac{\px\uedg}{\sqrt{\delta}}\right\vert\left\vert\sqrt{\delta}\px\Pedg\right\vert dx\\
\le& \frac{1}{\delta}\norm{\px\uedg(t,\cdot)}^2_{L^2(\R)} +\delta\norm{\px\Pedg(t,\cdot)}^2_{L^2(\R)}.
\end{align*}
Therefore, we get
\begin{equation*}
\frac{d}{dt}\left(\delta\norm{\px\Pedg(t,\cdot)}^2_{L^2(\R)}\right)\le \frac{1}{\delta}\norm{\px\uedg(t,\cdot)}^2_{L^2(\R)} +\delta\norm{\px\Pedg(t,\cdot)}^2_{L^2(\R)},
\end{equation*}
that is
\begin{equation*}
\frac{d}{dt}\left(\delta\norm{\px\Pedg(t,\cdot)}^2_{L^2(\R)}\right)-\delta\norm{\px\Pedg(t,\cdot)}^2_{L^2(\R)}\le \frac{1}{\delta}\norm{\px\uedg(t,\cdot)}^2_{L^2(\R)}.
\end{equation*}
The Gronwall Lemma and \eqref{eq:u0eps} give
\begin{equation}
\begin{split}
\label{eq:23}
\delta\norm{\px\Pedg(t,\cdot)}^2_{L^2(\R)}\le& C_{0}e^{t}+ \frac{e^{t}}{\delta}\int_{0}^{t}e^{-s}\norm{\px\uedg(s,\cdot)}^2_{L^2(\R)}ds\\
\le &C(T)+\frac{C(T)}{\delta}\int_{0}^{t}\norm{\px\uedg(s,\cdot)}^2_{L^2(\R)}ds.
\end{split}
\end{equation}
Due to \eqref{stima-l-2},
\begin{equation}
\label{eq:19}
\frac{1}{\delta} \int_{0}^{t}\norm{\px\uedg(s,\cdot)}^2_{L^2(\R)}ds= \frac{\eps}{\delta\eps}\int_{0}^{t}\norm{\px\uedg(s,\cdot)}^2_{L^2(\R)}ds \le \frac{C_{0}}{\delta\eps}.
\end{equation}
Since $0<\eps,\, \delta <1$, it follows from \eqref{eq:23} and \eqref{eq:19} that
\begin{equation*}
\delta\norm{\px\Pedg(t,\cdot)}^2_{L^2(\R)}\le C(T)\left(1+\frac{1}{\delta\eps}\right)\le C(T)\left(\frac{\delta\eps+1}{\delta\eps}\right)\le \frac{C(T)}{\delta\eps}.
\end{equation*}
Hence,
\begin{equation*}
\norm{\px\Pedg(t,\cdot)}^2_{L^2(\R)}\le \frac{C(T)}{\delta^2\eps},
\end{equation*}
which gives \eqref{eq:Px-in-l2}.\\
Let us show that \eqref{eq:p-l-infty} holds. We begin by observing that, thanks to the H\"older inequality,
\begin{equation}
\label{eq:holder}
\begin{split}
\Pedg^2(t,x)=& 2\int_{-\infty}^{x}\Pedg(t,y)\px\Pedg(t,y)dy\\
 \le &2\int_{\R}\vert \Pedg(t,y)\vert\vert\px\Pedg(t,y)dy\vert dx\\
 \le& \norm{\Pedg(t,\cdot)}_{L^{2}(\R)}\norm{\px\Pedg(t,\cdot)}_{L^{2}(\R)}.
\end{split}
\end{equation}
It follows from \eqref{stima-l-2} and \eqref{eq:Px-in-l2} that
\begin{equation*}
\norm{\Pedg}^2_{L^{\infty}((0,T)\times\R)} \le \frac{C_{0}}{\sqrt{\delta\gamma}}\frac{C(T)}{\delta\sqrt{\eps}}\le \frac{C(T)}{\delta^{\frac{3}{2}}\gamma^{\frac{1}{2}}\eps^{\frac{1}{2}}},
\end{equation*}
which gives \eqref{eq:p-l-infty}.
\end{proof}
\begin{lemma}\label{lm:u-l-infty}
Let $T>0$. Assume \eqref{eq:def-di-gamma}. Then, there exists  $C(T)>0$, independent on $\eps$, $\delta$ and $\gamma$, such that
\begin{equation}
\label{eq:u-in-l-infty}
\norm{\uedg}_{L^{\infty}((0,T)\times\R)}\le \norm{u_{0}}_{L^{\infty}(\R)}+ C(T).
\end{equation}
\end{lemma}
\begin{proof}
We begin by observing that, from \eqref{eq:def-di-gamma} and \eqref{eq:p-l-infty}, we have
\begin{equation*}
\pt \uedg +\px f(\uedg) -\eps\pxx\uedg \le \gamma\norm{\Pedg}_{L^{\infty}((0,T)\times\R)}\le \frac{\gamma^{\frac{3}{4}} C(T)}{\delta^{\frac{3}{4}}\eps^{\frac{1}{4}}}\le C(T).
\end{equation*}
Since the map
\begin{equation*}
{\mathcal F}(t):=\norm{u_{0}}_{L^\infty(\R)}+ C(T)t,
\end{equation*}
solves the equation
\begin{equation*}
\frac{d{\mathcal F}}{dt}=C(T)
\end{equation*}
and
\begin{equation*}
\max\{\uedg(0,x),0\}\le {\mathcal F}(t),\qquad (t,x)\in (0,T)\times\R,
\end{equation*}
the comparison principle for parabolic equations implies that
\begin{equation*}
 \uedg(t,x)\le {\mathcal F}(t),\qquad (t,x)\in (0,T)\times\R.
\end{equation*}
In a similar way, we can prove that
\begin{equation*}
\uedg(t,x)\ge -{\mathcal F}(t),\qquad (t,x)\in (0,T)\times\R.
\end{equation*}
Therefore,
\begin{equation*}
\vert\uedg(t,x)\vert\le\norm{u_{0}}_{L^\infty(\R)}+ C(T)t\le\norm{u_{0}}_{L^\infty(\R)}+ C(T),
\end{equation*}
which gives \eqref{eq:u-in-l-infty}.
\end{proof}
To prove Theorem \ref{th:main-1}, the following technical lemma is needed  \cite{Murat:Hneg}.
\begin{lemma}
\label{lm:1}
Let $\Omega$ be a bounded open subset of $
\R^2$. Suppose that the sequence $\{\mathcal
L_{n}\}_{n\in\mathbb{N}}$ of distributions is bounded in
$W^{-1,\infty}(\Omega)$. Suppose also that
\begin{equation*}
\mathcal L_{n}=\mathcal L_{1,n}+\mathcal L_{2,n},
\end{equation*}
where $\{\mathcal L_{1,n}\}_{n\in\mathbb{N}}$ lies in a
compact subset of $H^{-1}_{loc}(\Omega)$ and
$\{\mathcal L_{2,n}\}_{n\in\mathbb{N}}$ lies in a
bounded subset of $\mathcal{M}_{loc}(\Omega)$. Then $\{\mathcal
L_{n}\}_{n\in\mathbb{N}}$ lies in a compact subset of $H^{-1}_{loc}(\Omega)$.
\end{lemma}

Now, we are ready for the proof of Theorem \ref{th:main-1}.
\begin{proof}[Proof of Theorem \ref{th:main-1}]
Let $\eta:\R\to\R$ be any convex $C^2$ entropy function, and
$q:\R\to\R$ be the corresponding entropy
flux defined by $q'=f'\eta'$. By multiplying the first equation in \eqref{eq:OHepsw} with
$\eta'(\uedg)$ and using the chain rule, we get
\begin{equation}
\label{eq:30}
    \pt  \eta(\uedg)+\px q(\uedg)
    =\underbrace{\eps \pxx \eta(\uedg)}_{=:\CL_{1,\eps,\,\delta,\,\gamma}}
    \, \underbrace{-\eps \eta''(\uedg)\left(\px  \uedg\right)^2}_{=: \CL_{2,\eps,\,\delta,\gamma}}
     \, \underbrace{+\gamma\eta'(\uedg) \Pedg}_{=: \CL_{3,\eps,\,\delta,\,\gamma}},
\end{equation}
where  $\CL_{1,\eps,\,\delta,\gamma}$, $\CL_{2,\eps,\,\delta,\,\gamma}$, $\CL_{3, \eps,\,\delta,\,\gamma}$ are distributions.

Let us show that
\begin{equation*}
\label{eq:H1}
\textrm{$\CL_{1,\eps,\,\delta,\,\gamma}\to 0$ in $H^{-1}((0,T)\times\R)$, $T>0$.}
\end{equation*}
Since
\begin{equation*}
\eps\pxx\eta(\uedg)=\px(\eps\eta'(\uedg)\px\uedg),
\end{equation*}
from Lemmas \ref{lm:stima-l-2} and \ref{lm:u-l-infty},
\begin{align*}
\norm{\eps\eta'(\uedg)\px\uedg}^2_{L^2((0,T)\times (\R))}&\le\eps ^2\norm{\eta'}^2_{L^{\infty}(I_T)}\int_{0}^{T}\norm{\px\uedg(s,\cdot)}^2_{L^2(0,\infty)}ds\\
&\le\eps\norm{\eta'}^2_{L^{\infty}(I_T)}C_{0}\to 0,
\end{align*}
where
\begin{equation*}
I_T=\left(-\norm{u_{0}}_{L^{\infty}(\R)}- C(T),\norm{u_{0}}_{L^{\infty}(\R)}+ C(T)\right).
\end{equation*}
We claim that
\begin{equation*}
\label{eq:L1}
\textrm{$\{\CL_{2,\eps,\,\delta,\,\gamma}\}_{\eps,\,\delta,\,\gamma >0}$ is uniformly bounded in $L^1((0,T)\times\R))$, $T>0$}.
\end{equation*}
Again by Lemmas \ref{lm:stima-l-2} and \ref{lm:u-l-infty},
\begin{align*}
\norm{\eps\eta''(\uedg)(\px\uedg)^2}_{L^1((0,T)\times\R)}&\le
\norm{\eta''}_{L^{\infty}(I_T)}\eps
\int_{0}^{T}\norm{\px\uedg(s,\cdot)}^2_{L^2(\R)}ds\\
&\le \norm{\eta''}_{L^{\infty}(I_T)}C(T).
\end{align*}
We have that
\begin{equation*}
\textrm{$\{\CL_{3,\eps,\,\delta,\,\gamma}\}_{\eps,\,\delta>0}$ is uniformly bounded in $L^1_{loc}((0,T)\times (0,\infty))$, $T>0$.}
\end{equation*}
Let $K$ be a compact subset of $(0,T)\times (\R)$. From \eqref{eq:def-di-gamma} and \eqref{eq:p-l-infty},
\begin{align*}
\norm{\gamma\eta'(\uedg)\Pedg}_{L^1(K)}&=\gamma\int_{K}\vert\eta'(\uedg)\vert\vert\Pedg\vert
dtdx\\
&\leq \gamma
\norm{\eta'}_{L^{\infty}(I_T)}\norm{\Pe}_{L^{\infty}((0,T)\times\R)}\vert K \vert\\
&\le\frac{\gamma^{\frac{3}{4}}}{\delta^{\frac{3}{4}}\eps^{\frac{1}{4}}}C(T)\norm{\eta'}_{L^{\infty}(I_T)}\vert K \vert\\
&=C(T)\norm{\eta'}_{L^{\infty}(I_T)}\vert K \vert.
\end{align*}
Therefore, Lemma \ref{lm:1} implies that
\begin{equation}
\label{eq:GMC1}
    \text{$\left\{  \pt  \eta(\uedg)+\px q(\uedg)\right\}_{\eps,\,\delta,\,\gamma>0}$
    lies in a compact subset of $\Hneg((0,\infty)\times\R)$.}
\end{equation}
The $L^{\infty}$ bound stated in Lemma \ref{lm:u-l-infty}, \eqref{eq:GMC1} and the
 Tartar's compensated compactness method \cite{TartarI} give the existence of a subsequence
$\{\uedgk\}_{k\in\N}$ and a limit function $u\in L^{\infty}((0,T)\times\R)$
such that
\begin{equation}\label{eq:convu}
    \textrm{$\uedgk \to u$ a.e.~and in $L^{p}_{loc}((0,T)\times\R)$, $1\le p<\infty$}.
\end{equation}
Hence,
\begin{equation}
\label{eq:udelta-to-ueps}
 \textrm{$\uedgk \to u$  in $L^{\infty}((0,T)\times\R)$}.
\end{equation}
We conclude by proving that $u$ is unique entropy solution of \eqref{eq:cl}. Let $\phi\in C^{\infty}(\R^2)$ be a positive text function with compact support. We have to prove that
\begin{equation}
\label{eq:26}
\int_{0}^{\infty}\!\!\!\int_{\R}(\eta(u)\pt\phi +q(u)\px\phi)dtdx+\int_{\R}\eta\left(u_{0}(x)\right)\phi(0,x)dx\ge 0.
\end{equation}
From \eqref{eq:30}, we have
\begin{equation*}
\pt\eta(\uedgk) + \px q(\uedgk) \le \epsk\pxx\eta(\uedgk)+ \gk \eta'(\uedg) \Pedg.
\end{equation*}
Multiplying by $\phi$ and integrating on $(0,\infty)\times\R$, we have that
\begin{equation}
\label{eq:31}
\begin{split}
&\int_{0}^{\infty}\!\!\!\int_{\R}(\eta(\uedgk)\pt\phi +q(\uedgk)\px\phi)dtdx+\int_{\R}\eta\left(u_{0,\epsk,\,\dk,\,\gk}(x)\right)\phi(0,x)dx\\
&\qquad + \epsk\int_{0}^{\infty}\!\!\!\int_{\R}\eta(\uedgk)\pxx\phi dtdx+\gk\int_{0}^{\infty}\!\!\!\int_{\R}\eta'(\uedgk) \Pedgk\phi dtdx\ge 0.
\end{split}
\end{equation}
Let us show that
\begin{equation}
\label{eq:32}
\gk\int_{0}^{\infty}\!\!\!\int_{\R}\eta'(\uedgk) \Pedgk\phi dtdx\to 0.
\end{equation}
From \eqref{eq:def-di-gamma}, \eqref{eq:P-l-2-1}, \eqref{eq:u-in-l-infty} and the H\"older inequality, we get
\begin{align*}
&\gk\left\vert\int_{0}^{\infty}\!\!\!\int_{\R}\eta'(\uedgk) \Pedgk\phi dtdx\right\vert\\
&\quad\le \gk \int_{0}^{\infty}\!\!\!\int_{\R}\vert\eta'(\uedgk)\vert \vert\Pedgk\vert\vert\phi\vert dtdx\\
&\quad\le\gk\norm{\eta'}_{L^{\infty}(I_T)}\norm{\Pedgk}_{L^{2}(\supp(\phi))}\norm{\phi}_{L^2(\supp(\phi))}\\
&\quad \le \gk\norm{\eta'}_{L^{\infty}(I_T)}\norm{\Pedgk}_{L^{2}((0,T)\times\R)}\norm{\phi}_{L^2((0,T)\times\R)}\\
&\quad \le \frac{\gk^{\frac{1}{2}}}{\dk^{\frac{1}{2}}}C(T)\norm{\eta'}_{L^{\infty}(I_T)}\norm{\phi}_{L^2((0,T)\times\R)}\\
&\quad = \epsk^{\frac{1}{6}}C(T)\norm{\eta'}_{L^{\infty}(I_T)}\norm{\phi}_{L^2((0,T)\times\R)}\to 0,
\end{align*}
that is \eqref{eq:32}. 
Therefore, \eqref{eq:26} follows from \eqref{eq:u0eps}, \eqref{eq:u-in-l-infty}, \eqref{eq:31}, \eqref{eq:32} and the Lebesgue Dominated Convergence Theorem.

Finally, \eqref{u-media-nulla-1} and \eqref{eq:convu} give \eqref{eq:umedianulla}.
\end{proof}

\section{Ostrovsky equation: $\gamma\to 0$.}\label{sec:OE}
In this section, we consider the following Cauchy probelm
\begin{equation}
\label{eq:OE1}
\begin{cases}
u +\frac{1}{2}\px u^2 -\beta\pxxx u =\gamma P, &\quad t>0,\,x\in\R,\\
\px P=u  &\quad t>0,\,x\in\R,\\
P(t,-\infty)=0 &\quad t>0, \\
u(0,x)=u_0(x), &\quad x\in\R,
\end{cases}
\end{equation}
or equivalently,
\begin{equation}
\label{eq:OE2}
\begin{cases}
\pt u +\frac{1}{2}\px u^2 -\beta\pxxx u =\gamma \int_{-\infty}^{x} u(t,y) dy, &\quad t>0,\, x\in\R,\\
u(0,x)=u_0(x), &\quad x\in\R.
\end{cases}
\end{equation}
On the initial datum, we assume
\begin{equation}
\label{eq:assinit1}
u_0\in L^2(\R)\cap L^{4}(\R),\quad\int_{\R}u_{0}(x)dx=0,
\end{equation}
and on the function
\begin{equation}
\label{eq:def-di-P01}
P_{0}(x)=\int_{-\infty}^{x} u_{0}(y)dy, \quad x\in\R,
\end{equation}
we assume that
\begin{equation}
\label{eq:L-2P01}
\int_{\R}P_0(x)dx= \int_{\R}\left(\int_{-\infty}^{x}u_{0}(y)dy\right)dx=0.
\end{equation}
We observe that, if $\beta, \gamma \to 0$, then \eqref{eq:OE1} reads
\begin{equation}
\label{eq:Bu1}
\begin{cases}
\pt u +\frac{1}{2}\px u^2=0,&\quad t>0,\, x\in\R,\\
u(0,x)=u_0(x), &\quad x\in\R.
\end{cases}
\end{equation}
which is the Burges' equation.

Fix four small numbers $0<\eps, \beta,\delta,\gamma<1 $, and let $\uebdg=\uebdg(t,x)$ be the unique
classical solution of the following mixed problem:
\begin{equation}
\label{eq:OHepswb}
\begin{cases}
\pt \uebdg+\frac{1}{2} \px \uebdg^2 - \beta\pxxx\uebdg\\
\qquad=\gamma\Pebdg+ \eps\pxx\uebdg, &\quad t>0,\, x\in\R,\\
-\delta\pt\Pebdg +\px\Pebdg=\uebdg,&\quad t>0, x\in\R,\\
\Pebdg(t,-\infty)=0, &\quad t>0,\\
\uebdg(0,x)=u_{\eps,\beta, \delta, \gamma, 0}(x),&\quad x\in\R,
\end{cases}
\end{equation}
where $u_{\eps,\beta,\delta, \gamma, 0}$ is a $C^\infty$ approximation of $u_{0}$ such that
\begin{equation}
\label{eq:u0epsbeta}
\begin{split}
&u_{\eps,\,\beta,\,\delta,\,\gamma,0} \to u_{0} \quad  \textrm{in $L^{p}_{loc}(\R)$, $1\le p < 4$, as $\eps,\,\beta,\,\delta,\,\gamma \to 0$,}\\
&\norm{u_{\eps,\,\beta,\,\delta,\,\gamma,0}}^2_{L^2(\R)}+\delta\gamma\norm{P_{\eps,\,\beta,\,\delta,\,\gamma,0}}^2_{L^2(\R)}+
 \delta\norm{\px P_{\eps,\,\beta,\,\delta,\gamma,0}}^2_{L^2(\R)}\\
&\qquad\qquad\quad\quad\quad+\norm{u_{\eps,\,\beta,\,\delta,\,\gamma,0}}^4_{L^4(\R)}+(\beta+\eps^2)\norm{\px u_{\eps,\,\beta,\,\delta,\,\gamma,0}}^2_{L^2(\R)}\\
&\qquad\qquad\quad\quad\quad+\beta^2\norm{\pxx u_{\eps,\,\beta,\,\delta,\,\gamma,0}}^2_{L^2(\R)}\le C_0,\quad \eps,\beta,\delta,\gamma >0,\\
&\beta\int_{\R}u_{\eps,\,\beta,\,\delta,\,\gamma,0}(\px u_{\eps,\,\beta,\,\delta,\,\gamma,0})^2 \le C_0, \quad \eps,\,\beta,\,\delta,\,\gamma >0,\\
&\int_{\R}u_{\eps,\,\beta,\,\delta,\,\gamma,0}(x) dx =0,\quad \int_{\R}P_{\eps,\,\beta,\,\delta,\,\gamma,0}(x) dx =0, \quad \eps,\,\beta,\,\delta,\,\gamma>0,
\end{split}
\end{equation}
and $C_0$ is a constant independent on $\eps$,$\beta$, $\delta$ and $\gamma$.

The main result of this section is the following theorem.
\begin{theorem}
\label{th:main-3}
Assume that \eqref{eq:assinit1}, \eqref{eq:def-di-P01}, \eqref{eq:L-2P01},  and \eqref{eq:u0epsbeta} hold.
If
\begin{equation}
\label{eq:beta-eps}
\beta=\mathbf{\mathcal{O}}(\eps^2), \quad \gamma=\mathbf{\mathcal{O}}(\eps\delta)
\end{equation}
then, there exist four sequences $\{\eps_{k}\}_{k\in\N}$, $\{\beta_{k}\}_{k\in\N}$, $\{\delta_{k}\}_{k\in\N}$, $\{\gamma_{k}\}_{k\in\N}$ with $\eps_k, \beta_k, \delta_k, \gamma_k \to 0$, and a limit function $u\in L^{\infty}(0,T; L^4(\R)\cap L^2(\R)),\ T>0$, such that
\begin{itemize}
\item[$i)$] $u_{\eps_k, \beta_k, \delta_k, \gamma_k}\to u$ strongly in $L^{p}_{loc}((0,T)\times\R)$, for each $1\le p <4$, $T>0$,
\end{itemize}
and $u$ is a distributional solution of \eqref{eq:Bu1}. Moreover, if
\begin{equation}
\label{eq:beta-eps-1}
\beta=o(\eps^2),\quad \gamma=\mathbf{\mathcal{O}}(\eps\delta)
\end{equation}
then,
\begin{itemize}
\item[$ii)$] $u$ is  the unique entropy solution of \eqref{eq:Bu1}.
\end{itemize}
In particular, we have \eqref{eq:umedianulla}.
\end{theorem}
Let us prove some a priori estimates on $\uebdg$ and $\Pebdg$, denoting with $C_0$ the constants which depend on the initial datum, and $C(T)$ the constants which depend also on $T$.

Arguing as Section \ref{sec:OH}, we obtain the following results
\begin{lemma}\label{lm:we1}
For each $t>0$,
\begin{align}
\label{eq:P-pxP-intfy-2}
\Pebdg(t,\infty)=\px\Pebdg(t,-\infty)&=\px\Pebdg(t,\infty)=0,\\
\label{eq:up-2}
\int_{\R}\uebdg(t,x)dx &= -\delta\frac{d}{dt}\int_{\R}\Pebdg(t,x)dx,\\
\label{eq:Pmedianilla-1}
\int_{\R}\Pebdg(t,x)dx&=0,\\
\label{u-media-nulla}
\int_{\R}\uebdg(t,x)dx&=0.
\end{align}
In particular, we have that
\begin{equation}
\label{eq:stima-l-2-1}
\norm{\uebdg(t,\cdot)}^2_{L^2(\R)} + \delta\gamma\norm{\Pebdg(t,\cdot)}^2_{L^2(\R)} +2\eps\int_{0}^{t}\norm{\px\uebdg(s,\cdot)}^2_{L^2(\R)}dx\le C_{0}.
\end{equation}
Moreover, fixed $T>0$, there exists  $C(T)>0$, independent on $\eps$, $\beta$, $\delta$ and $\gamma$, such that,
\begin{align}
\label{eq:P-l-2-2}
\norm{\Pebdg}_{L^{2}((0,T)\times\R)}&\le \frac{C(T)}{\delta^{\frac{1}{2}}\gamma^{\frac{1}{2}}},\\
\label{eq:Px-in-l21}
\norm{\px\Pebdg(t,\cdot)}_{L^2(\R)}&\le \frac{C(T)}{\delta\sqrt{\eps}},\\
\label{eq:p-l-infty-1}
\norm{\Pebdg}_{L^{\infty}((0,T)\times\R)}&\le \frac{C(T)}{\delta^{\frac{3}{4}}\gamma^{\frac{1}{4}}\eps^{\frac{1}{4}}},
\end{align}
for every $0\le t\le T.$
\end{lemma}
\begin{lemma}\label{lm:12}
Fixed $T>0$. Then,
\begin{equation}
\label{eq:u-infty-2}
\norm{\uebdg}_{L^{\infty}((0,T)\times\R)}\le C(T)\beta^{-\frac{1}{3}}.
\end{equation}
Moreover, for every $0\le t\le T$,
\begin{equation}
\label{eq;px-u-l2-1}
\beta\norm{\uebdg(t,\cdot)}^2_{L^2(\R)}+ \beta\eps \int_{0}^{t}\norm{\pxx\uebdg(s,\cdot)}^2_{L^2(\R)}ds \le C(T)\beta^{-\frac{1}{3}}.
\end{equation}
\end{lemma}

\begin{proof}
Let $0\le t\le T$. Multiplying  \eqref{eq:OHepswb} by $-2\beta\pxx\uebdg + \uebdg^2$, and arguing as \cite[Lemma $2.5$]{Cd2}, we obtain that
\begin{equation}
\label{eq:1002}
\begin{split}
&\frac{d}{dt}\left(\beta\norm{\px\uebdg(t,\cdot)}^2_{L^2(\R)}+\frac{1}{3}\int_{\R}\uebdg^3dx\right) \\
& \qquad+ 2\eps\int_{\R} \uebdg(\px\uebdg)^2dx+ 2\beta\eps\norm{\pxx\uebdg(t,\cdot)}^2_{L^2(\R)}\\
&\quad = -2\gamma\beta\int_{\R}\pxx\uebdg\Pebdg dx  + \gamma\int_{\R}\uebdg^2\Pebdg dx.
\end{split}
\end{equation}
Since $0 <\beta,\eps < 1$, it follows from \eqref{eq:beta-eps}, \eqref{eq:stima-l-2-1} and the Young inequality that
\begin{equation}
\label{eq:1011}
\begin{split}
&2\gamma\beta\int_{\R}\pxx\uebdg\Pebdg dx\le 2\int_{\R}\left\vert\beta\sqrt{\eps}\pxx\uebdg\right\vert\left\vert\frac{\gamma}{\sqrt{\eps}}\Pebdg\right\vert dx\\
&\quad \le \beta^2\eps\norm{\pxx\uebdg(t,\cdot)}^2_{L^2(\R)}+ \frac{\gamma^2}{\eps} \norm{\Pebdg(t,\cdot)}^2_{L^2(\R)}\\
&\quad \le \beta\eps \norm{\pxx\uebdg(t,\cdot)}^2_{L^2(\R)}+ \frac{\gamma^2}{\delta\gamma}C(T)\\
&\quad \le \beta\eps \norm{\pxx\uebdg(t,\cdot)}^2_{L^2(\R)}+C(T)\eps\\
&\quad \le \beta\eps \norm{\pxx\uebdg(t,\cdot)}^2_{L^2(\R)}+C(T).
\end{split}
\end{equation}
Since $0< \delta,\eps <1$, due to \eqref{eq:beta-eps}, \eqref{eq:stima-l-2-1} and the H\"older inequality,
\begin{equation}
\label{eq:1012}
\begin{split}
&\gamma\int_{\R}\uebdg^2\Peb dx\le \gamma\int_{\R} \uebdg^2 \vert \Pebdg\vert dx\\
&\quad \le \gamma\norm{\uebdg}_{L^{\infty}((0,T)\times\R)}\int_{\R} \vert \uebdg\vert\vert \Pebdg\vert dx\\
&\quad \le \frac{\gamma}{\sqrt{\delta}\sqrt{\gamma}}C_{0} \norm{\uebdg}_{L^{\infty}((0,T)\times\R)}\le\eps^{\frac{1}{2}}C(T)\norm{\uebdg}_{L^{\infty}((0,T)\times\R)}\\
&\quad  \le C(T)\norm{\uebdg}_{L^{\infty}((0,T)\times\R)}.
\end{split}
\end{equation}
Therefore, \eqref{eq:stima-l-2-1}, \eqref{eq:1002}, \eqref{eq:1011} and \eqref{eq:1012} give
\begin{align*}
&\frac{d}{dt}\left(\beta\norm{\px\uebdg(t,\cdot)}^2_{L^2(\R)}+\frac{1}{3}\int_{\R}\uebdg^3dx\right)+ \beta\eps\norm{\pxx\uebdg(t,\cdot)}^2_{L^2(\R)}\\
&\quad \le C(T)\norm{\uebdg}_{L^{\infty}((0,T)\times\R)}- 2\eps\int_{\R} \uebdg(\px\uebdg)^2dx +C(T)\\
&\quad \le C(T)\norm{\uebdg}_{L^{\infty}((0,T)\times\R)}+2\eps \int_{\R} \vert\uebdg\vert(\px\uebdg)^2dx +C(T)\\
&\quad \le C(T)\norm{\uebdg}_{L^{\infty}((0,T)\times\R)}+2\eps \norm{\uebdg}_{L^{\infty}((0,T)\times\R)}\int_{\R}(\px\uebdg)^2dx+C(T).
\end{align*}
It follows from \eqref{eq:u0epsbeta}, \eqref{eq:stima-l-2-1} and an integration on $(0,t)$ that
\begin{align*}
&\beta\norm{\px\uebdg(t,\cdot)}^2_{L^2(\R)} + \beta\eps \int_{0}^{t}\norm{\pxx\uebdg(s,\cdot)}^2_{L^2(\R)} ds\\
&\quad \le C_{0}+ C(T)\norm{\uebdg}_{L^{\infty}((0,T)\times\R)}\int_{0}^t ds+C(T) \int_{0}^{t}ds \\
&\qquad+2\eps \norm{\uebdg}_{L^{\infty}((0,T)\times\R)}\int_{0}^{t}\norm{\px\uebdg(s,\cdot)}^2_{L^2(\R)} ds +\frac{1}{3}\int_{\R}\vert\uebdg\vert^3 dx\\
&\quad \le C(T) + C(T)\norm{\uebdg}_{L^{\infty}((0,T)\times\R)} + C_{0}C(T)\norm{\uebdg}_{L^{\infty}((0,T)\times\R)}\\
&\qquad +\frac{1}{3}\norm{\uebdg}_{L^{\infty}((0,T)\times\R)}\norm{\uebdg(t,\cdot)}^2_{L^2(\R)}\\
&\quad \le C(T) + C(T)\norm{\uebdg}_{L^{\infty}((0,T)\times\R)}+\frac{1}{3}C_0 \norm{\uebdg}_{L^{\infty}((0,T)\times\R)}.
\end{align*}
Therefore,
\begin{equation}
\label{eq:stima-100}
\begin{split}
\beta\norm{\px\uebdg(t,\cdot)}^2_{L^2(\R)}&+ \beta\eps \int_{0}^{t}\norm{\pxx\uebdg(s,\cdot)}^2_{L^2(\R)}ds\\
\le & C(T)\left(1+\norm{\uebdg}_{L^{\infty}((0,T)\times\R)}\right).
\end{split}
\end{equation}
Due to \eqref{eq:stima-l-2-1}, \eqref{eq:stima-100} and the H\"older inequality,
\begin{align*}
\uebdg^2(t,x)&=2\int_{-\infty}^{x}\uebdg\px\uebdg dy\le 2\int_{\R}\vert\uebdg\px\uebdg\vert dx\\
&\le\frac{2}{\sqrt{\beta}}\norm{\uebdg}_{L^2(\R)}\sqrt{\beta}\norm{\px\uebdg(t,\cdot)}_{L^2(\R)}\\
&\le\frac{2}{\sqrt{\beta}}C_0\sqrt{ C(T)\left(1+\norm{\uebdg}_{L^{\infty}((0,T)\times\R)}\right)},
\end{align*}
that is
\begin{equation}
\label{eq:quarto-grado}
\norm{\uebdg}^4_{L^{\infty}((0,T)\times\R)} \le \frac{C(T)}{\beta}\left( 1+\norm{\uebdg}_{L^{\infty}((0,T)\times\R)}\right).
\end{equation}
Arguing as \cite[Lemma $2.7$]{Cd2}, we have \eqref{eq:u-infty-2}.

Finally, \eqref{eq;px-u-l2-1} follows from \eqref{eq:u-infty-2} and \eqref{eq:stima-100}.
\end{proof}

\begin{lemma}\label{lm:13}
Let  $T>0$. Assume \eqref{eq:beta-eps} holds true. Then:
\begin{itemize}
\item[$i$)] the family $\{\uebdg\}_{\eps,\beta,\delta,\gamma}$ is bounded in $L^{4}((0,T)\times\R)$;
\item[$ii$)] the following families $\{\beta\pxx\uebdg\}_{\eps,\beta,\delta,\gamma},\,\{\sqrt{\eps}\uebdg\px\uebdg\}_{\eps,\beta,\delta,\beta},$\\ $\{\beta\sqrt{\eps}\pxxx\uebdg \}_{\eps,\beta,\delta,\gamma} $ are bounded in $L^2((0,T)\times\R)$.
 \end{itemize}
\end{lemma}
The proof of the previous lemma is based on the regularity of the functions $\uebdg$ and \cite[Lemma $2.5$]{Cd2}.

\begin{proof}[Proof of Lemma \ref{lm:13}]
Let $0\le t \le T$. Multiplying  \eqref{eq:OHepswb} by
\begin{equation*}
\uebdg^3 - 3\beta(\px\uebdg)^2 - 6\beta\uebdg\pxx\ueb +\frac{18}{5}\beta^2\pxxxx\uebdg,
\end{equation*}
and arguing as \cite[Lemma $2.6$]{Cd2}, we obtain that
\begin{align*}
\frac{d}{dt} G(t) &+3D_1\eps \int_{\R}\uebdg^2(\px\uebdg)^2 dx +\eps\beta^2D_2\int_{\R}(\pxxx\uebdg)^2dx\\
\le & \gamma \int_{\R} \vert\uebdg\vert^3 \vert\Pebdg\vert dx +3 \gamma \beta\int_{\R}(\px\uebdg)^2\vert \Pebdg\vert dx\\
&+6 \gamma\beta \int_{\R} \vert\uebdg\vert \vert \pxx\uebdg\vert \vert \Pebdg\vert dx +\frac{18}{5}\gamma\beta^2 \int_{\R}\pxxxx\uebdg\Pebdg dx,
\end{align*}
where
\begin{equation}
\label{eq:def-di-G}
G(t)=\frac{1}{4}\int_{\R}\uebdg^4 dx + 3\beta\int_{\R}\ueb(\px\uebdg)^2 dx + \frac{9}{5}\beta^2\int_{\R}(\pxx\uebdg)^2 dx,
\end{equation}
while $D_1$, $D_2$ are fixed positive constants.

Due to \eqref{eq:P-pxP-intfy-2},
\begin{equation}
\label{eq:1212}
\frac{18}{5}\gamma\beta^2 \int_{\R}\pxxxx\uebdg\Pebdg dx=-\frac{18}{5}\gamma\beta^2 \int_{\R}\pxxx\uebdg\px\Pebdg dx
\end{equation}
Since $0 <\beta< 1$, it follows from \eqref{eq:beta-eps}, \eqref{eq:Px-in-l21}, \eqref{eq:1212} and the Young inequality that
\begin{align*}
&-\frac{18}{5}\gamma\beta^2 \int_{\R}\pxxx\uebdg\px\Pebdg dx\le  \frac{18}{5}\gamma\beta^2\left\vert \int_{\R}\pxxx\uebdg\px\Pebdg dx\right\vert\\
&\quad \le\frac{18}{5} \int_{\R}\left\vert\beta^2\sqrt{A\eps}\pxxx\uebdg\right\vert\left\vert\frac{\gamma\px\Pebdg}{\sqrt{A\eps}}\right\vert\\
&\quad \le\frac{9 A}{5}\beta^4\eps\int_{\R}(\pxxx\uebdg)^2 dx+ \frac{9}{5 A}\frac{\gamma^2}{\eps}\int_{\R}(\px\Pebdg)^2 dx\\
&\quad \le\frac{9 A}{5}\beta^2\eps\int_{\R}(\pxxx\uebdg)^2 dx+\frac{9}{5 A}\frac{\gamma^2}{\delta^2\eps^2}C(T)\\
&\quad \le\frac{9 A}{5}\beta^2\eps\int_{\R}(\pxxx\uebdg)^2 dx+\frac{C(T)}{A},
\end{align*}
where $A$ is a positive constant that will be specified later.
Therefore,
\begin{align*}
\frac{d}{dt} G(t) &+3D_1\eps \int_{\R}\uebdg^2(\px\uebdg)^2 dx+\eps\beta^2\left (D_2- \frac{9 A}{5}\right)\int_{\R}(\pxxx\uebdg)^2dx\\
\le & \gamma \int_{\R} \vert\uebdg\vert^3 \vert\Pebdg\vert dx +3 \gamma \beta\int_{\R}(\px\uebdg)^2\vert \Pebdg\vert dx\\
&+6 \gamma\beta \int_{\R} \vert\uebdg\vert \vert \pxx\uebdg\vert \vert \Pebdg\vert dx +\frac{C(T)}{A}.
\end{align*}
Choosing $ \displaystyle A < \frac{5 D_2}{9}$, we have
\begin{align*}
\frac{d}{dt} G(t) &+3D_1\eps \int_{\R}\uebdg^2(\px\uebdg)^2 dx+\eps\beta^2D_3\int_{\R}(\pxxx\uebdg)^2dx\\
\le & \gamma \int_{\R} \vert\uebdg\vert ^3 \vert\Pebdg\vert dx +3 \gamma \beta\int_{\R}(\px\uebdg)^2\vert \Pebdg\vert dx\\
&+6 \gamma\beta \int_{\R} \vert\uebdg\vert \vert \pxx\uebdg\vert \vert \Pebdg\vert dx +C(T),
\end{align*}
where $D_3$ is a fixed positive constant.

Since $0<\eps<1$, due to \eqref{eq:beta-eps}, \eqref{eq:stima-l-2-1}, \eqref{eq:p-l-infty-1} and the Young inequality, we obtain that
\begin{align*}
&\gamma\int_{\R}\vert \uebdg^3\vert\vert\Pebdg\vert dx dx+6\gamma\beta\int_{\R}\vert\uebdg\vert\vert\pxx\uebdg\vert\vert\Pebdg\vert\\
&\quad =\int_{\R}\left\vert\frac{1}{\sqrt{2}}\uebdg^2\right\vert\left\vert \sqrt{2}\gamma\uebdg\Pebdg\right\vert dx\\
&\qquad+ \int_{\R}\left\vert\frac{3\sqrt{2}}{\sqrt{5}}\beta\pxx\uebdg\right\vert \left\vert \frac{\sqrt{5}\gamma}{\sqrt{2}}\uebdg\Pebdg\right\vert dx\\
&\quad \leq\frac{1}{4}\int_{\R}\uebdg^4 dx + \gamma^2\int_{\R}\uebdg^2\Peb^2 dx\\
&\qquad+\frac{9}{5}\beta^2\int_{\R}(\pxx\uebdg)^2 dx+ \frac{5\gamma^2}{4}\int_{\R}\uebdg^2\Pebdg^2dx\\
&\quad=  \frac{1}{4}\int_{\R}\uebdg^4 dx + \frac{9}{5}\beta^2\int_{\R}(\pxx\uebdg)^2 dx + \frac{9\gamma^2}{4}\int_{\R}\uebdg^2\Pebdg^2dx\\
&\quad \le \frac{1}{4}\int_{\R}\uebdg^4 dx + \frac{9}{5}\beta^2\int_{\R}(\pxx\uebdg)^2 dx\\
&\qquad+ \frac{9\gamma^2}{4} \norm{\Pebdg}^2_{L^{\infty}((0,T)\times\R)}\int_{\R}\uebdg^2(t,x) dx \\
&\quad \le \frac{1}{4}\int_{\R}\uebdg^4 dx + \frac{9}{5}\beta^2\int_{\R}(\pxx\uebdg)^2 dx +C(T)\frac{\gamma^2}{\delta^{\frac{3}{2}}\gamma^{\frac{1}{2}}\eps^{\frac{1}{2}}}\\
&\quad\leq  \frac{1}{4}\int_{\R}\uebdg^4 dx + \frac{9}{5}\beta^2\int_{\R}(\pxx\uebdg)^2 dx +\eps C(T)\\
&\quad \le \frac{1}{4}\int_{\R}\uebdg^4 dx + \frac{9}{5}\beta^2\int_{\R}(\pxx\uebdg)^2 dx + C(T).
\end{align*}
Again by \eqref{eq:beta-eps} and \eqref{eq:p-l-infty-1},
\begin{align*}
3 \gamma \beta\int_{\R}(\px\uebdg)^2\vert \Pebdg\vert dx\le&3 \gamma \beta \norm{\Pebdg}_{L^{\infty}((0,T)\times\R)}\int_{\R}(\px\uebdg)^2 dx\\
\le &C(T)\frac{\gamma}{{\delta^{\frac{3}{4}}\gamma^{\frac{1}{4}}}\eps^{\frac{1}{4}}}\beta\int_{\R}(\px\uebdg)^2 dx\\
\le& C(T)\sqrt{\eps}\beta\int_{\R}(\px\uebdg)^2 dx\le C(T)\beta\int_{\R}(\px\uebdg)^2 dx.
\end{align*}
Hence,
\begin{align*}
\frac{d}{dt}G(t)&+ 3 \eps D_{1}\int_{\R}\ueb^2(\px\uebdg)^2 dx + \eps\beta^2 D_{3}\int_{\R}(\pxxx\uebdg)^2 dx\\
\le & \frac{1}{4}\int_{\R}\uebdg^4 dx + \frac{9}{5}\beta^2\int_{\R}(\pxx\uebdg)^2 dx + C(T)\beta\int_{\R} (\px\uebdg)^2 dx + C(T).
\end{align*}
Arguing as \cite[Lemma $2.6$]{Cd2}, thanks to \cite[Lemma $2.7$]{Cd2}, we have
\begin{align*}
\norm{\uebdg}_{L^4((0,T)\times\R)}\le& C(T),\\
\beta\norm{\pxx\uebdg}_{L^2((0,T)\times\R)}\le& C(T),\\
\sqrt{\eps}\norm{\uebdg\px\uebdg}_{L^2((0,T)\times\R)}\le& C(T),\\
\sqrt{\eps}\beta\norm{\pxxx\uebdg}_{L^2((0,T)\times\R)}\le& C(T).
\end{align*}
The proof is done.
\end{proof}

\begin{lemma}\label{lm:15}
Let  $T>0$. Assume that \eqref{eq:beta-eps} holds true. Then:
\begin{itemize}
\item[$i$)] the family $\{\eps\px\uebdg\}_{\eps,\beta,\delta,\gamma}$ is bounded in $L^{\infty}(0,T;L^2(\R))$;
\item[$ii$)] the family $\{\eps\sqrt{\eps}\pxx\uebdg\}_{\eps,\beta,\delta,\gamma}$ is bounded in $L^2((0,T)\times\R)$;
\item[$iii$)] the family $\{\beta\px\ueb\pxx\uebdg\}_{\eps,\beta,\delta,\gamma}$ is bounded in $L^1((0,T)\times\R)$.
\end{itemize}
Moreover,
\begin{equation}
\label{eq:defuxx}
\beta^2\int_{0}^{T}\norm{\pxx\uebdg(s,\cdot)}^2_{L^2(\R)}ds \le C(T)\eps.
\end{equation}
\end{lemma}

\begin{proof}
Let $0\le t\le T$. Multiplying \eqref{eq:OHepswb} by $-\eps^2\pxx\uebdg$, arguing as \cite[Lemma $2.8$]{Cd2}, we have
\begin{align*}
\eps^2\frac{d}{dt}\int_{\R}(\px\uebdg)^2dx &+2\eps^3\int_{\R}(\pxx\uebdg)^2dx\\
=& 2\eps^2 \int_{\R}\ueb\px\uebdg\pxx\uebdg dx -2\eps^2\gamma \int_{\R}\Pebdg\pxx\uebdg dx.
\end{align*}
Since $0<\eps< 1$, due to \eqref{eq:beta-eps}, \eqref{eq:stima-l-2-1} and the Young inequality,
\begin{align*}
&2\eps^2 \int_{\R}\ueb\px\uebdg\pxx\uebdg dx -2\eps^2\gamma \int_{\R}\Pebdg\pxx\uebdg dx\\
&\quad\le \left\vert 2\eps^2 \int_{\R}\ueb\px\uebdg\pxx\uebdg dx -2\eps^2\gamma \int_{\R}\Pebdg\pxx\uebdg dx\right\vert\\
&\quad \le 2\eps^2\left \vert  \int_{\R}\ueb\px\uebdg\pxx\uebdg dx\right\vert + 2\eps^2\gamma \left\vert\int_{\R}\Pebdg\pxx\uebdg dx\right\vert\\
&\quad \le 2\int_{\R}\eps^{\frac{1}{2}}\vert\ueb\px\uebdg\vert\eps^{\frac{3}{2}}\vert\pxx\uebdg\vert dx + \int_{\R}2\eps^{\frac{1}{2}}\gamma\vert \Pebdg\vert \eps^{\frac{3}{2}}\vert\pxx\uebdg\vert dx\\
&\quad\le \eps\int_{\R}(\ueb\px\uebdg)^2 dx + \eps^3 \int_{\R} (\pxx\uebdg)^2 dx+2\gamma^2\eps\int_{\R} \Pebdg^2 dx \\
&\qquad+\frac{\eps^3}{2}\int_{\R}(\pxx\uebdg)^2 dx\\
&\quad \le \eps\int_{\R}(\ueb\px\uebdg)^2 dx+\frac{3\eps^3}{2}\int_{\R}(\pxx\uebdg)^2dx + \frac{\gamma^2}{\delta\gamma}C(T)\\
&\quad \le \eps\int_{\R}(\ueb\px\uebdg)^2 dx+\frac{3\eps^3}{2}\int_{\R}(\pxx\uebdg)^2dx+C(T).
\end{align*}
Thus,
\begin{equation*}
\eps^2\frac{d}{dt}\int_{\R}(\px\uebdg)^2dx+\frac{\eps^3}{2}\int_{\R}(\pxx\uebdg)^2dx\le \eps\int_{\R}(\ueb\px\uebdg)^2 dx +C(T).
\end{equation*}
An integration on $(0,t)$, \eqref{eq:u0epsbeta} and Lemma \ref{lm:13} give
\begin{equation*}
\eps^2\norm{\px\uebdg(t,\cdot)}^2_{L^2(\R)} +\frac{\eps^3}{2}\int_{0}^{t} \norm{\pxx\uebdg(s,\cdot)}^2_{L^2(\R)} ds \le C(T).
\end{equation*}
Hence,
\begin{equation}
\label{eq:10032}
\begin{split}
\eps^2\norm{\px\uebdg(t, \cdot)}^2_{L^2(\R)}&\le C(T),\\
\eps^3\int_{0}^{t}\norm{\pxx\uebdg(s,\cdot)}^2_{L^2(\R)}ds &\le C(T).
\end{split}
\end{equation}
Thanks to  \eqref{eq:beta-eps}, \eqref{eq:stima-l-2-1}, \eqref{eq:10032} and the H\"older inequality,
\begin{align*}
&\beta\int_{0}^{T}\!\!\!\int_{\R}\vert\px\uebdg\pxx\uebdg\vert dsdx =\frac{\beta}{\eps^2}\int_{0}^{T}\!\!\!\int_{\R}\eps^{\frac{1}{2}}\vert\px\uebdg\vert\eps^{\frac{3}{2}}\vert\pxx\uebdg\vert dx\\
&\quad \le \frac{\beta}{\eps^2} \left(\eps \int_{0}^{T}\!\!\!\int_{\R}(\px\uebdg)^2 dsdx\right)^{\frac{1}{2}}\left(\eps^3 \int_{0}^{T}\!\!\!\int_{\R}(\pxx\uebdg)^2 dsdx\right)^{\frac{1}{2}}\\
&\quad \le C_0 C(T)\frac{\beta}{\eps^2}\le C(T).
\end{align*}
Due to \eqref{eq:beta-eps} and \eqref{eq:10032}, we have
\begin{align*}
\beta^2\int_{0}^{T}\norm{\pxx\uebdg(s,\cdot)}^2_{L^2(\R)}ds \le C_{0}^2\eps^4\int_{0}^{T}\norm{\pxx\uebdg(s,\cdot)}^2_{L^2(\R)}ds\le C(T)\eps,
\end{align*}
which gives \eqref{eq:defuxx}.
\end{proof}
To prove Theorem \ref{th:main-3}, we use Lemma \ref{lm:1} and the following definition.
\begin{definition}
A pair of functions $(\eta, q)$ is called an  entropy--entropy flux pair if $\eta :\R\to\R$ is a $C^2$ function and $q :\R\to\R$ is defined by
\begin{equation*}
q(u)=\int^{u} \eta'(\xi)f'(\xi)d\xi.
\end{equation*}
An entropy-entropy flux pair $(\eta,\, q)$ is called  convex/compactly supported if, in addition, $\eta$ is convex/compactly supported.
\end{definition}
We begin by proving the following result.
\begin{lemma}\label{lm:dist-solution}
Assume that \eqref{eq:assinit1}, \eqref{eq:def-di-P01}, \eqref{eq:L-2P01}, \eqref{eq:u0epsbeta}, and \eqref{eq:beta-eps} hold. Then for any compactly supported entropy--entropy flux pair $(\eta,\, q)$, there exist four sequences $\{\eps_{k}\}_{k\in\N}$, $\{\beta_{k}\}_{k\in\N}$, $\{\delta_{k}\}_{k\in\N}$, $\{\gamma_{k}\}_{k\in\N}$,  with $\eps_k, \beta_k, \delta_k, \gamma_k \to 0$,  and a limit function
\begin{equation*}
u\in L^{\infty}(0,T; L^2(\R)\cap L^4(\R)),\qquad T>0
\end{equation*}
 such that
\begin{equation}
\label{eq:con1}
 u_{\eps_k, \beta_k, \delta_k, \gamma_k}\to u \quad  \textrm{in} \quad  L^{p}_{loc}((0,\infty)\times\R),\quad \textrm{for each} \quad 1\le p <4,
\end{equation}
and $u$ is a distributional solution of \eqref{eq:Bu1}.
\end{lemma}
\begin{proof}
Let us consider a compactly supported entropy--entropy flux pair $(\eta, q)$. Multiplying \eqref{eq:OHepswb} by $\eta'(\uebdg)$, we have
\begin{align*}
\pt\eta(\uebdg) + \px q(\uebdg) =&\eps \eta'(\uebdg) \pxx\uebdg + \beta \eta'(\uebdg) \pxxx\uebdg\\
& + \gamma \eta'(\uebdg) \Pebdg\\
=& I_{1,\,\eps,\,\beta,\,\delta,\,\gamma}+I_{2,\,\eps,\,\beta,\,\delta,\,\gamma}+ I_{3,\,\eps,\,\beta,\,\delta,\,\gamma} + I_{4,\,\eps,\,\beta,\,\delta,\,\gamma}+I_{5,\,\eps,\,\beta,\,\delta,\,\gamma},
\end{align*}
where
\begin{equation}
\begin{split}
\label{eq:12000}
I_{1,\,\eps,\,\beta,\,\delta,\,\gamma}&=\px(\eps\eta'(\uebdg)\px\ueb),\\
I_{2,\,\eps,\,\beta,\,\delta,\,\gamma}&= -\eps\eta''(\uebdg)(\px\uebdg)^2,\\
I_{3,\,\eps,\,\beta,\,\delta,\,\gamma}&= \px(\beta\eta'(\uebdg)\pxx\uebdg),\\
I_{4,\,\eps,\,\beta,\,\delta,\,\gamma}&= -\beta\eta''(\uebdg)\px\ueb\pxx\uebdg,\\
I_{5,\,\eps,\,\beta,\,\delta,\,\gamma}&=\gamma\eta'(\uebdg) \Pebdg.
\end{split}
\end{equation}
Arguing as \cite[Lemma $3.2$]{Cd2}, we have that  $I_{1,\,\eps,\,\beta,\,\delta,\,\gamma}\to0$ in $H^{-1}((0,T) \times\R)$, $\{I_{2,\,\eps,\,\beta,\,\delta,\,\gamma}\}_{\eps,\beta,\delta,\gamma >0}$ is bounded in $L^1((0,T)\times\R)$, $I_{3,\,\eps,\,\beta,\,\delta,\,\gamma}\to0$ in $H^{-1}((0,T) \times\R)$, $\{I_{4,\,\eps,\,\beta,\,\delta,\, \gamma}\}_{\eps,\beta,\delta,\gamma >0}$ is bounded in $L^1((0,T)\times\R)$.

Let us  show that
\begin{equation*}
I_{5,\,\eps,\,\beta,\,\delta,\,\gamma}\to 0\quad  \text{in $L^1_{loc}((0,\infty)\times\R)$}, \quad\text{as}\quad \eps\to 0.
\end{equation*}
Let $K$ be a compact subset of $(0,T)\times\R$. \eqref{eq:beta-eps} and Lemma \ref{lm:we1} give
\begin{align*}
\norm{\gamma\eta'(\uebdg)\Pebdg}_{L^1(K)}&=\gamma\int_{K}\vert\eta'(\uebdg)\vert\vert\Pebdg\vert
dtdx\\
&\leq \gamma\norm{\eta'}_{L^{\infty}(\R)}\norm{\Pebdg}_{L^{\infty}((0,T)\times\R)}\vert K \vert\\
&\le \frac{\gamma}{\delta^{\frac{3}{4}}\gamma^{\frac{1}{4}}\eps^{\frac{1}{4}}}C(T)\norm{\eta'}_{L^{\infty}(\R)}\vert K \vert\\
&\le\sqrt{\eps}C(T)\norm{\eta'}_{L^{\infty}(\R)}\vert K \vert\to 0.
\end{align*}
Therefore, Lemma \ref{lm:1} and the $L^p$ compensated compactness of \cite{SC} give \eqref{eq:con1}.

We conclude by proving that $u$ is a distributional solution of \eqref{eq:Bu1}.
Let $ \phi\in C^{\infty}(\R^2)$ be a test function with
compact support. We have to prove that
\begin{equation}
\label{eq:k1}
\int_{0}^{\infty}\!\!\!\!\!\int_{\R}\left(u\pt\phi+\frac{u^2}{2}\px\phi\right)dtdx +\int_{\R}u_{0}(x)\phi(0,x)dx=0.
\end{equation}
We have that
\begin{align*}
\int_{0}^{\infty}\!\!\!&\!\!\int_{\R}\left(u_{\eps_{k}, \beta_{k}, \delta_{k},\gamma_{k}}\pt\phi+\frac{u^2_{\eps_k, \beta_{k}, \delta_{k},\gamma_{k}}}{2}\px\phi\right)dtdx
+\int_{\R}u_{0,\eps_k,\beta_k,\delta_k,\gamma_k}(x)\phi(0,x)dx  \\
&- \gamma_{k}\int_{0}^{\infty}\!\!\!\!\!\int_{\R} P_{\eps_{k},\beta_{k},\delta_k,\gamma_k}\phi dtdx+\eps_{k}\int_{0}^{\infty}\!\!\!\!\!\int_{\R}u_{\eps_{k},\beta_{k}, \delta_{k},\gamma_k}\pxx\phi dtdx \\
&+ \eps_k\int_{0}^{\infty}u_{0,\eps_{k},\beta_{k},\delta_k,\gamma_k}(x)\pxx\phi(0,x)dx- \beta_k\int_{0}^{\infty}\!\!\!\!\int_{\R}u_{\eps_k,\beta_k,\delta_k,\gamma_k}\pxxx\phi dt dx\\
&- \beta_k\int_{0}^{\infty}u_{0,\eps_k,\beta_k,\delta_k,\gamma_k}(x)\pxxx\phi(0,x)dx=0.
\end{align*}
Let us show that
\begin{equation}
\label{eq:p-to-0}
- \gamma_{k}\int_{0}^{\infty}\!\!\!\!\!\int_{\R} P_{\eps_{k},\beta_{k},\delta_k,\gamma_k}\phi dtdx\to 0.
\end{equation}
From \eqref{eq:beta-eps} and \eqref{eq:p-l-infty-1}, we get
\begin{align*}
&\gk\left\vert\int_{0}^{\infty}\!\!\!\int_{\R} \Pebdgk\phi dtdx\right\vert\\
&\quad\le \gk \int_{0}^{\infty}\!\!\!\int_{\R}\vert\Pebdgk\vert\vert\phi\vert dtdx\\
&\quad\le\gk\norm{\Pebdgk}_{L^{\infty}((0,T)\times\R)}\int_{0}^{\infty}\!\!\!\int_{\R}\vert\phi\vert dtdx \\
&\quad \le\frac{\gk}{\delta_{k}^{\frac{3}{4}}\gk^{\frac{1}{4}}\eps_k^{\frac{1}{4}}}C(T)\int_{0}^{\infty}\!\!\!\int_{\R}\vert\phi\vert dtdx \\
&\quad \le \sqrt{\eps_{k}}C(T)\int_{0}^{\infty}\!\!\!\int_{\R}\vert\phi\vert dtdx\to 0,
\end{align*}
that is \eqref{eq:p-to-0}.

Therefore, \eqref{eq:k1} follows from \eqref{eq:u0epsbeta}, \eqref{eq:con1} and \eqref{eq:p-to-0}.
\end{proof}
Arguing as \cite{LN}, we prove the following result.
\begin{lemma}
\label{lm:entropy-solution}
Assume that \eqref{eq:assinit1}, \eqref{eq:def-di-P01}, \eqref{eq:L-2P01}, \eqref{eq:u0epsbeta}, and \eqref{eq:beta-eps-1} hold. Then,
\begin{equation}
\label{eq:con3}
u_{\eps_k, \beta_k, \delta_k \gamma_k}\to u \quad  \textrm{in} \quad  L^{p}_{loc}((0,\infty)\times\R),\quad \textrm{for each} \quad 1\le p <4,
\end{equation}
where $u$ is  the unique entropy solution of \eqref{eq:Bu1}.
\end{lemma}
\begin{proof}
Let us consider a compactly supported entropy--entropy flux pair $(\eta, q)$. Multiplying \eqref{eq:OHepswb} by $\eta'(\uebdg)$, we obtain that
\begin{align*}
\pt\eta(\uebdg) + \px q(\uebdg) =&\eps \eta'(\uebdg) \pxx\uebdg + \beta \eta'(\uebdg) \pxxx\uebdg\\
& + \gamma \eta'(\uebdg) \Pebdg\\
=& I_{1,\,\eps,\,\beta,\,\delta,\,\gamma}+I_{2,\,\eps,\,\beta,\,\delta,\,\gamma} I_{3,\,\eps,\,\beta,\,\delta,\,\gamma} + I_{4,\,\eps,\,\beta,\,\delta,\,\gamma}+I_{5,\,\eps,\,\beta,\,\delta,\,\gamma},
\end{align*}
where $I_{1,\,\eps,\,\beta,\,\delta,\,\gamma},\,I_{2,\,\eps,\,\beta,\,\delta,\,\gamma},\, I_{3,\,\eps,\,\beta,\,\delta,\,\gamma} ,\, I_{4,\,\eps,\,\beta,\,\delta,\,\gamma}$ and $I_{5,\,\eps,\,\beta,\,\delta,\,\gamma}$ are defined in \eqref{eq:12000}.

Arguing as \cite[Lemma $3.3$]{Cd2}, we obtain that $I_{1,\,\eps,\,\beta,\,\delta,\,\gamma}\to0$ in $H^{-1}((0,T) \times\R)$,\\ $\{I_{2,\,\eps,\,\beta,\,\delta,\,\gamma}\}_{\eps,\beta,\delta,\gamma >0}$ is bounded in $L^1((0,T)\times\R)$, $I_{3,\,\eps,\,\beta,\,\delta,\,\gamma}\to0$ in $H^{-1}((0,T) \times\R)$, $I_{4,\,\eps,\,\beta,\,\delta,\,\gamma}\to 0$ in $L^1((0,T)\times\R)$, while arguing in Lemma \ref{lm:dist-solution}, $I_{5,\,\eps,\,\beta,\,\delta,\,\gamma}\to 0$ in $L_{loc}^1((0,\infty)\times\R)$.

Therefore, Lemma \ref{lm:1} gives \eqref{eq:con3}.

We conclude by proving that $u$ is the unique entropy solution of \eqref{eq:Bu1}.
Let us consider a compactly supported entropy--entropy flux pair $(\eta, q)$, and $\phi\in C^{\infty}_{c}((0,\infty)\times\R)$  non--negative. We have to prove that
\begin{equation}
\label{eq:u-entropy-solution}
\int_{0}^{\infty}\!\!\!\!\!\int_{\R}(\pt\eta(u)+ \px q(u))\phi dtdx \le0.
\end{equation}
We have that
\begin{align*}
&\int_{0}^{\infty}\!\!\!\!\!\int_{\R}(\px\eta(u_{\eps_{k},\,\beta_{k},\,\delta_{k},\,\gamma_{k}})+\px q(u_{\eps_{k},\,\beta_{k},\,\delta_{k},\,\gamma_{k}}))\phi dtdx \\
&\qquad= \gamma_{k}\int_{0}^{\infty}\!\!\!\!\!\int_{\R}P_{\eps_{k},\,\beta_{k},\,\delta_{k},\,\gamma_{k}}\eta'(u_{\eps_{k},\,\beta_{k},\,\delta_{k},\,\gamma_{k}})\phi dtdx\\
&\qquad\quad +\eps_{k}\int_{0}^{\infty}\!\!\!\!\!\int_{\R}\px(\eta'(u_{\eps_{k},\,\beta_{k},\,\delta_{k},\,\gamma_{k}})\px u_{\eps_{k},\,\beta_{k},\,\delta_{k},\,\gamma_{k}})\phi dtdx\\
&\qquad\quad -\eps_{k}\int_{0}^{\infty}\!\!\!\!\!\int_{\R} \eta''(u_{\eps_{k},\,\beta_{k},\,\delta_{k},\,\gamma_{k}})(\px u_{\eps_{k},\,\beta_{k},\,\delta_{k},\,\gamma_{k}})^2\phi dtdx\\
&\qquad\quad+\beta_{k}\int_{0}^{\infty}\!\!\!\!\!\int_{\R}\px(\eta'(u_{\eps_{k},\,\beta_{k},\,\delta_{k},\,\gamma_{k}})\pxx u_{\eps_{k},\,\beta_{k},\,\delta_{k},\,\gamma_{k}})\phi dtdx\\
&\qquad\quad  +\beta_{k}\int_{0}^{\infty}\!\!\!\!\!\int_{\R}\eta''(u_{\eps_{k},\,\beta_{k},\delta_{k},\,\gamma_{k}})\px u_{\eps_{k},\,\beta_{k},\,\delta_{k},\,\gamma_{k}}\pxx u_{\eps_{k},\,\beta_{k},\,\delta_{k},\,\gamma_{k}}\phi dtdx\\
&\qquad \le \gamma_{k}\int_{0}^{\infty}\!\!\!\!\!\int_{\R}P_{\eps_{k},\,\beta_{k},\,\delta_{k},\,\gamma_{k}}\eta'(u_{\eps_{k},\,\beta_{k},\,\delta_{k},\,\gamma_{k}})\phi dtdx\\
&\qquad\quad - \eps_{k}\int_{0}^{\infty}\!\!\!\!\!\int_{\R}\eta'(u_{\eps_{k},\,\beta_{k},\,\delta_{k},\,\gamma_{k}})\px u_{\eps_{k},\,\beta_{k},\,\delta_{k},\,\gamma_{k}}\px\phi dtdx \\
&\qquad\quad- \beta_{k}\int_{0}^{\infty}\!\!\!\!\!\int_{\R}\eta'(u_{\eps_{k},\,\beta_{k},\,\delta_{k},\,\gamma_{k}})\pxx u_{\eps_{k},\,\beta_{k},\,\delta_{k},\,\gamma_{k}}\px\phi dtdx\\
&\qquad\quad- \beta_{k}\int_{0}^{\infty}\!\!\!\!\!\int_{\R}\eta''(u_{\eps_{k},\,\beta_{k},\,\delta_{k},\,\gamma_{k}})\px u_{\eps_{k},\,\beta_{k},\,\delta_{k},\,\gamma_{k}}\pxx u_{\eps_{k},\,\beta_{k},\,\delta_{k},\,\gamma_{k}}\phi dtdx\\
&\qquad \le \gamma_{k}\int_{0}^{\infty}\!\!\!\!\!\int_{\R}\vert P_{\eps_{k},\,\beta_{k},\,\delta_{k},\,\gamma_{k}}\vert\vert\eta'(u_{\eps_{k},\,\beta_{k},\,\delta_{k},\,\gamma_{k}})\vert\vert\phi\vert dtdx\\ &\qquad\quad+\eps_{k}\int_{0}^{\infty}\!\!\!\!\!\int_{\R}\vert\eta'(u_{\eps_{k},\,\beta_{k},\,\delta_{k},\,\gamma_{k}})\vert\vert\px u_{\eps_{k},\,\beta_{k},\,\delta_{k},\,\gamma_{k}}\vert\vert\px\phi\vert dtdx \\
&\qquad\quad +\beta_{k}\int_{0}^{\infty}\!\!\!\!\!\int_{\R}\vert\eta'(u_{\eps_{k},\,\beta_{k},\,\delta_{k},\,\gamma_{k}})\vert\vert\pxx u_{\eps_{k},\,\beta_{k},\,\delta_{k},\,\gamma_{k}}\vert\vert\px\phi\vert dtdx\\ &\qquad\quad+\beta_{k}\int_{0}^{\infty}\!\!\!\!\!\int_{\R}\vert\eta''(u_{\eps_{k},\,\beta_{k},\,\delta_{k},\,\gamma_{k}})\vert\vert\px u_{\eps_{k},\,\beta_{k},\,\delta_{k},\,\gamma_{k}}\pxx u_{\eps_{k},\,\beta_{k},\,\delta_{k},\,\gamma_{k}}\vert\vert\phi\vert dtdx\\
&\qquad\le \gamma_{k} \norm{\eta'}_{L^{\infty}(\R)}\norm{P_{\eps_{k},\,\beta_{k},\,\delta_{k},\,\gamma_{k}}}_{L^2(supp(\px\phi))}\norm{\px\phi}_{L^2(supp(\px\phi))}\\
&\qquad\quad+ \eps_{k} \norm{\eta'}_{L^{\infty}(\R)}\norm{\px u_{\eps_{k},\,\beta_{k},\,\delta_{k},\,\gamma_{k}}}_{L^2(supp(\px\phi))}\norm{\px\phi}_{L^2(supp(\px\phi))}\\
&\qquad\quad+ \beta_{k} \norm{\eta'}_{L^{\infty}(\R)}\norm{\pxx u_{\eps_{k},\,\beta_{k},\,\delta_{k},\,\gamma_{k}}}_{L^2(supp(\px\phi))}\norm{\px\phi}_{L^2(supp(\px\phi))}\\
&\qquad\quad +\beta_{k} \norm{\eta''}_{L^{\infty}(\R)}\norm{\phi}_{L^{\infty}(\R^{+}\times\R)}\norm{\px u_{\eps_{k},\,\beta_{k},\,\delta_{k},\,\gamma_{k}}\pxx u_{\eps_{k},\,\beta_{k},\,\delta_{k},\,\gamma_{k}}}_{L^1(supp(\px\phi))}\\
&\qquad\le \gamma_{k} \norm{\eta'}_{L^{\infty}(\R)}\norm{P_{\eps_{k},\,\beta_{k},\,\delta_{k},\,\gamma_{k}}}_{L^2((0,T)\times\R))}\norm{\px\phi}_{L^2((0,T)\times\R))}\\
&\qquad\quad+\eps_{k} \norm{\eta'}_{L^{\infty}(\R)}\norm{\px u_{\eps_{k},\,\beta_{k},\,\delta_{k},\,\gamma_{k}}}_{L^2((0,T)\times\R)}\norm{\px\phi}_{L^2((0,T)\times\R)}\\
&\qquad\quad+ \beta_{k} \norm{\eta'}_{L^{\infty}(\R)}\norm{\pxx u_{\eps_{k},\,\beta_{k},\,\delta_{k},\,\gamma_{k}}}_{L^2((0,T)\times\R)}\norm{\px\phi}_{L^2((0,T)\times\R)}\\
&\qquad\quad+\beta_{k} \norm{\eta''}_{L^{\infty}(\R)}\norm{\phi}_{L^{\infty}(\R^{+}\times\R)}\norm{\px u_{\eps_{k},\,\beta_{k},\,\delta_{k},\,\gamma_{k}}\pxx u_{\eps_{k},\,\beta_{k},\,\delta_{k},\,\gamma_{k}}}_{L^1((0,T)\times\R)}.
\end{align*}
\eqref{eq:u-entropy-solution} follows from \eqref{eq:beta-eps-1}, \eqref{eq:con3}, Lemmas \ref{lm:we1} and \ref{lm:15}.
\end{proof}
\begin{proof}[Proof of Theorem \ref{th:main-3}]
Theorem \ref{th:main-3} follows from Lemmas \ref{lm:dist-solution}, and \ref{lm:entropy-solution}, while \eqref{eq:umedianulla} follows from \eqref{u-media-nulla} \eqref{eq:con1}, or \eqref{eq:con3}. Therefore, the proof is done.
\end{proof}

\section{The regularized short pulse  equation: $\gamma\to 0$.}\label{sec:RSPE}
In this section, we consider the following Cauchy probelm
\begin{equation}
\label{eq:OE3}
\begin{cases}
u -\frac{1}{6}\px u^3 -\beta\pxxx u =\gamma P, &\quad t>0,\,x\in\R,\\
\px P=u  &\quad t>0,\,x\in\R,\\
P(t,-\infty)=0 &\quad t>0, \\
u(0,x)=u_0(x), &\quad x\in\R,
\end{cases}
\end{equation}
or equivalently,
\begin{equation}
\label{eq:OE4}
\begin{cases}
\pt u -\frac{1}{6}\px u^3 -\beta\pxxx u =\gamma \int_{-\infty}^{x} u(t,y) dy, &\quad t>0,\, x\in\R,\\
u(0,x)=u_0(x), &\quad x\in\R.
\end{cases}
\end{equation}
On the initial datum, we assume
\begin{equation}
\label{eq:assinit2}
u_0\in L^2(\R)\cap L^{6}(\R),\quad\int_{\R}u_{0}(x)dx=0,
\end{equation}
and on the function
\begin{equation}
\label{eq:def-di-P02}
P_{0}(x)=\int_{-\infty}^{x} u_{0}(y)dy, \quad x\in\R,
\end{equation}
we assume that
\begin{equation}
\label{eq:L-2P02}
\int_{\R}P_0(x)dx= \int_{\R}\left(\int_{-\infty}^{x}u_{0}(y)dy\right)dx=0.
\end{equation}
We observe that, if $\beta, \gamma \to 0$, then \eqref{eq:OE3} reads
\begin{equation}
\label{eq:Bu2}
\begin{cases}
\pt u -\frac{1}{6}\px u^3=0,&\quad t>0,\, x\in\R,\\
u(0,x)=u_0(x), &\quad x\in\R.
\end{cases}
\end{equation}
Fix four small numbers $0<\eps,\beta,\delta,\gamma<1 $, and let $\uebdg=\uebdg(t,x)$ be the unique
classical solution of the following mixed problem:
\begin{equation}
\label{eq:OHepswb-1}
\begin{cases}
\pt \uebdg-\frac{1}{6} \px \uebdg^3 - \beta\pxxx\uebdg\\
\qquad=\gamma\Pebdg+ \eps\pxx\uebdg, &\quad t>0,\, x\in\R,\\
-\delta\pt\Pebdg +\px\Pebdg=\uebdg,&\quad t>0, x\in\R,\\
\Pebdg(t,-\infty)=0, &\quad t>0,\\
\uebdg(0,x)=u_{\eps,\beta, \delta, \gamma, 0}(x),&\quad x\in\R,
\end{cases}
\end{equation}
where $u_{\eps,\beta,\delta, \gamma, 0}$ is a $C^\infty$ approximation of $u_{0}$ such that
\begin{equation}
\label{eq:u0epsbeta-1}
\begin{split}
&u_{\eps,\,\beta,\,\delta,\,\gamma,0} \to u_{0} \quad  \textrm{in $L^{p}_{loc}(\R)$, $1\le p < 6$, as $\eps,\,\beta,\,\delta,\,\gamma \to 0$,}\\
&\norm{u_{\eps,\,\beta,\,\delta,\,\gamma,0}}^2_{L^2(\R)}+\delta\gamma\norm{P_{\eps,\,\beta,\,\delta,\,\gamma,0}}^2_{L^2(\R)}+ \delta\norm{\px P_{\eps,\,\beta,\,\delta,\,\gamma,0}}^2_{L^2(\R)}\\
&\qquad\qquad\quad\quad\quad+\norm{u_{\eps,\,\beta,\,\delta,\,\gamma,0}}^6_{L^6(\R)}+(\beta+\eps^2)\norm{\px u_{\eps,\,\beta,\,\delta,\,\gamma,0}}^2_{L^2(\R)}\\
&\qquad\qquad\quad\quad\quad+\beta^2\norm{\pxx u_{\eps,\,\beta,\,\delta,\,\gamma,0}}^2_{L^2(\R)}\le C_0,\quad \eps,\,\beta,\,\delta,\,\gamma >0,\\
&\int_{\R}u_{\eps,\,\beta,\,\delta,\,\gamma,0}(x) dx =0,\quad \int_{\R}P_{\eps,\,\beta,\,\delta,\,\gamma,0}(x) dx =0, \quad \eps,\,\beta,\,\delta,\,\gamma>0,
\end{split}
\end{equation}
and $C_0$ is a constant independent on $\eps$,$\beta$, $\delta$ and $\gamma$.

The main result of this section is the following theorem.
\begin{theorem}
\label{th:main-4}
Assume that \eqref{eq:assinit2}, \eqref{eq:def-di-P02}, \eqref{eq:L-2P02},  and \eqref{eq:u0epsbeta-1} hold.
If
\begin{equation}
\label{eq:beta-eps-2}
\beta=\mathbf{\mathcal{O}}(\eps^2), \quad \gamma=\mathbf{\mathcal{O}}(\eps\delta)
\end{equation}
then, there exist four sequences $\{\eps_{k}\}_{k\in\N}$, $\{\beta_{k}\}_{k\in\N}$, $\{\delta_{k}\}_{k\in\N}$, $\{\gamma_{k}\}_{k\in\N}$ with $\eps_k, \beta_k, \delta_k, \gamma_k \to 0$, and a limit function $u\in L^{\infty}(0,T; L^2(\R)\cap L^6(\R)),\ T>0$, such that
\begin{itemize}
\item[$i)$] $u_{\eps_k, \beta_k, \delta_k, \gamma_k}\to u$ strongly in $L^{p}_{loc}((0,T)\times\R)$, for each $1\le p <6$, $T>0$,
\end{itemize}
where $u$ is a distributional solution of \eqref{eq:Bu2}. Moreover, if
\begin{equation}
\label{eq:beta-eps-3}
\beta=o(\eps^2),\quad \gamma=\mathbf{\mathcal{O}}(\eps\delta)
\end{equation}
then,
\begin{itemize}
\item[$ii)$] $u$ is  the unique entropy solution of \eqref{eq:Bu2}.
\end{itemize}
In particular, we have \eqref{eq:umedianulla}.
\end{theorem}
Let us prove some a priori estimates on $\uebdg$ and $\Pebdg$, denoting with $C_0$ the constants which depend on the initial datum, and $C(T)$ the constants which depend also on $T$.

We begin by observing that Lemma \ref{lm:we1} holds  also for \eqref{eq:OHepswb-1}.
\begin{lemma}
Fixed $T>0$. There exists  $C(T)>0$, independent on $\eps$, $\beta$, $\delta$ and $\gamma$ such that
\begin{equation}
\label{eq:l-infty-u-1}
\norm{\uebdg}_{L^{\infty}((0,T)\times\R)}\le C(T)\beta^{-\frac{1}{2}}.
\end{equation}
Moreover, for every $0\le t\le T$,
\begin{equation}
\label{eq:ux21}
\beta\norm{\px\uebdg(t,\cdot)}^2_{L^2(\R)} + \beta\eps \int_{0}^{t}\norm{\pxx\uebdg(s,\cdot)}^2_{L^2(\R)}ds\le C(T)\beta^{-2}.
\end{equation}
\end{lemma}
\begin{proof}
Let $0\le t\le T$. Multiplying  \eqref{eq:OHepswb-1} by $\displaystyle -\beta\pxx\ueb -\frac{1}{6} \ueb^3$, we have
\begin{equation}
\label{eq:Ohmp-1}
\begin{split}
\left(-\beta\pxx\uebdg -\frac{1}{6}\uebdg^3\right)\pt\uebdg &- \frac{1}{6}\left(-\beta\pxx\uebdg -\frac{1}{6}\uebdg^3\right)\px\ueb^3\\
&-\beta\left(-\beta\pxx\uebdg -\frac{1}{6}\uebdg^3\right)\pxxx\uebdg\\
=&\gamma\left(-\beta\pxx\uebdg -\frac{1}{6} \ueb^3\right)\Pebdg\\
&+\eps\left(-\beta\pxx\uebdg -\frac{1}{6}\uebdg^3\right)\pxx\uebdg.
\end{split}
\end{equation}
Arguing as \cite[Lemma $2.3$]{Cd4}, we have
\begin{align*}
&\frac{d}{dt}\left(\beta\norm{\px\uebdg(t,\cdot)}^2_{L^2(\R)}-\frac{1}{12}\int_{\R}\uebdg^4dx\right)\\
&\quad\qquad + 2\beta\eps\norm{\pxx\uebdg(t,\cdot)}^2_{L^2(\R)}\\
&\quad= 2\gamma\beta\int_{\R}\pxx\uebdg\Pebdg dx -\frac{\gamma}{3}\int_{\R} \uebdg^3 \Pebdg dx\\
&\quad\qquad +\eps\int_{\R}\uebdg^2(\px\uebdg)^2dx.
\end{align*}
Since $0<\eps, \beta<1$, it follows from \eqref{eq:stima-l-2-1}, \eqref{eq:beta-eps-2} and the Young inequality that
\begin{align*}
2\gamma\beta\left\vert \int_{\R}\pxx\uebdg\Pebdg dx \right\vert \le& 2\int_{\R}\left\vert\beta\sqrt{\eps}\pxx\uebdg\right\vert\left\vert\frac{\gamma}{\sqrt{\eps}}\Pebdg\right\vert dx\\
\le & \eps\beta^2 \norm{\pxx\uebdg(t,\cdot)}^2_{L^2(\R)} + \frac{\gamma^2}{\eps} \norm{\Pebdg(t,\cdot)}^2_{L^2(\R)}\\
 \le& \eps\beta\norm{\pxx\uebdg(t,\cdot)}^2_{L^2(\R)} + \frac{\gamma}{\eps\delta}C(T)\\
 \le& \eps\beta\norm{\pxx\uebdg(t,\cdot)}^2_{L^2(\R)} +C(T).
\end{align*}
Moreover, from \eqref{eq:stima-l-2-1}, \eqref{eq:p-l-infty-1}, \eqref{eq:beta-eps-2} and the Young inequality, we have
\begin{align*}
\frac{\gamma}{3}\left\vert\int_{\R} \uebdg^3 \Pebdg dx\right\vert \le &\int_{\R} \left\vert\frac{\uebdg}{3}\right\vert \left\vert\gamma\uebdg^2\Pebdg\right\vert dx \\
\le & \frac{1}{6}\norm{\uebdg(t,\cdot)}^2_{L^{2}(\R)} + \frac{\gamma^2}{2}\int_{\R} \uebdg^4 \Pebdg^2 dx \\
\le & C(T) + \frac{\gamma^2}{2}\norm{\Pebdg}^2_{L^{\infty}((0,T)\times\R)} \norm{\uebdg}^2_{L^{\infty}((0,T)\times\R)} \norm{\uebdg(t,\cdot)}^2_{L^{2}(\R)}\\
\le & \frac{\gamma}{3} \norm{\Pebdg}_{L^{\infty}((0,T)\times\R)} \norm{\uebdg}_{L^{\infty}((0,T)\times\R)} \norm{\uebdg(t,\cdot)}^2_{L^{2}(\R)}\\
\le & C(T) +\frac{\gamma^2}{\delta^{\frac{3}{2}} \gamma^{\frac{1}{2} } \eps^{\frac{1}{2}}} C(T)\norm{\uebdg}^2_{L^{\infty}((0,T)\times\R)}\\
\le & C(T) +\eps C(T)\norm{\uebdg}^2_{L^{\infty}((0,T)\times\R)}\\
\le & C(T) + C(T)\norm{\uebdg}^2_{L^{\infty}((0,T)\times\R)}.
\end{align*}
Therefore,
\begin{align*}
&\frac{d}{dt}\left(\beta\norm{\px\uebdg(t,\cdot)}^2_{L^2(\R)}-\frac{1}{12}\int_{\R}\uebdg^4dx\right) + \beta\eps\norm{\pxx\uebdg(t,\cdot)}^2_{L^2(\R)}\\
&\quad \le C(T) + C(T)\norm{\uebdg}^2_{L^{\infty}((0,T)\times\R)} + \eps\int_{\R}\uebdg^2(\px\uebdg)^2dx\\
&\quad \le C(T) +C(T)\norm{\uebdg}^2_{L^{\infty}((0,T)\times\R)} +\eps \norm{\uebdg}^2_{L^{\infty}((0,T)\times\R)}\int_{\R}(\px\uebdg)^2dx.
\end{align*}
\eqref{eq:stima-l-2-1}, \eqref{eq:u0epsbeta-1} and an integration on $(0,t)$ gives
\begin{align*}
&\beta\norm{\px\uebdg(t,\cdot)}^2_{L^2(\R)}+\beta\eps\int_{0}^{t}\norm{\pxx\uebdg(s,\cdot)}^2_{L^2(\R)}ds\\
&\quad \le C_{0}+C(T)\norm{\uebdg}^2_{L^{\infty}((0,T)\times\R)}\int_{0}^{t}ds+C(T)\int_{0}^{t} ds +\frac{1}{12}\int_{\R}\uebdg^4dx\\
&\qquad\quad+\eps\norm{\uebdg}^2_{L^{\infty}((0,T)\times\R)}\int_{0}^{t}\norm{\px\uebdg(s,\cdot)}^2_{L^2(\R)}ds\\
&\quad\le C(T) +  C(T)\norm{\uebdg}^2_{L^{\infty}((0,T)\times\R)} + \frac{1}{12}\norm{\uebdg}^2_{L^{\infty}((0,T)\times\R)}\norm{\uebdg(s,\cdot)}^2_{L^2(\R)}\\    &\quad \le C(T)+C(T)\norm{\uebdg}^2_{L^{\infty}((0,T)\times\R)}+ \frac{C_0}{12}\norm{\uebdg}^2_{L^{\infty}((0,T)\times\R)},
\end{align*}
that is
\begin{equation}
\label{eq:ux12}
\begin{split}
\beta\norm{\px\uebdg(t,\cdot)}^2_{L^2(\R)}&+ \beta\eps\int_{0}^{t}\norm{\pxx\uebdg(s,\cdot)}^2_{L^2(\R)}ds\\
\le & C(T)\left(1+\norm{\uebdg}^2_{L^{\infty}((0,T)\times\R)}\right).
\end{split}
\end{equation}
Due to \eqref{eq:stima-l-2-1}, \eqref{eq:ux12} and the H\"older inequality,
\begin{align*}
\uebdg^2(t,x)&=2\int_{-\infty}^{x}\uebdg\px\uebdg dy\le 2\int_{\R}\vert\uebdg\px\uebdg\vert dx\\
&\le\frac{2}{\sqrt{\beta}}\norm{\uebdg}_{L^2(\R)}\sqrt{\beta}\norm{\px\uebdg(t,\cdot)}_{L^2(\R)}\\
&\le\frac{2}{\sqrt{\beta}}C_0\sqrt{ C(T)\left(1+\norm{\uebdg}^2_{L^{\infty}((0,T)\times\R)}\right)},
\end{align*}
that is
\begin{equation}
\label{eq:quarto-grado-1}
\norm{\uebdg}^4_{L^{\infty}((0,T)\times\R)} \le \frac{C(T)}{\beta}\left( 1+\norm{\uebdg}^2_{L^{\infty}((0,T)\times\R)}\right).
\end{equation}
Arguing as \cite[Lemma $2.3$]{Cd4}, we have \eqref{eq:l-infty-u-1}.

Finally, \eqref{eq:ux21} follows from \eqref{eq:l-infty-u-1} and \eqref{eq:ux12}.
\end{proof}

\begin{lemma}\label{lm:23}
Let $T>0$. Assume \eqref{eq:beta-eps-2} holds true. There exists $C(T)>0$, independent on $\eps$, $\beta$, $\delta$, and $\gamma$ such that
\begin{align}
\label{eq:u-in-l-6}
\norm{\uebdg(t,\cdot)}_{L^{6}(\R)}\le & C(T),\\
\label{eq:eps-px-u-l-21}
\eps\norm{\px\uebdg(t,\cdot)}_{L^{2}(\R)}\le & C(T),\\
\label{eq:p001}
\eps e^{t}\int_{0}^{t}\!\!\!\int_{\R}e^{-s}\uebdg^4(s,\cdot)(\px\uebdg(s,\cdot))^2 dsdx\le &C(T),\\
\label{eq:p002}
\eps^3 e^{t}\int_{0}^{t} e^{-s} \norm{\pxx\uebdg(s,\cdot)}^2_{L^2(\R)}ds \le & C(T),
\end{align}
for every $0<t<T$. Moreover,
\begin{align}
\label{eq:000101}
\beta\norm{\px\uebdg\pxx\uebdg}_{L^{1}((0,T)\times\R)} \le &C(T),\\
\label{eq:defuxx}
\beta^2\int_{0}^{T}\norm{\pxx\uebdg(s,\cdot)}^2_{L^2(\R)}ds \le& C(T)\eps.
\end{align}
\end{lemma}
\begin{proof}
Let $0\le t \le T$. Multiplying \eqref{eq:OHepswb-1} by $\ueb^5 -3\eps^2\pxx\ueb$, we have
\begin{align*}
(\uebdg^5 -3\eps^2\pxx\uebdg)\pt\uebdg &- \frac{1}{6}(\uebdg^5 -3\eps^2\pxx\uebdg)\px \uebdg^3\\
&-(\uebdg^5 -3\eps^2\pxx\uebdg)\beta\pxxx\uebdg\\
=&\gamma(\uebdg^5 -3\eps^2\pxx\uebdg)\Pebdg\\
& +\eps(\uebdg^5 -3\eps^2\pxx\uebdg)\pxx\uebdg.
\end{align*}
Arguing as \cite[Lemma $2.4$]{Cd4}, we get
\begin{equation}
\label{eq:Rhs1}
\begin{split}
&\frac{d}{dt}\left(\frac{1}{6}\norm{\uebdg(t,\cdot)}^6_{L^6(\R)}+\frac{3\eps^2}{2}\norm{\px\uebdg(t,\cdot)}^2_{L^2(\R)}\right)\\
&\qquad + 5\eps\int_{\R}\ueb^4(\px\uebdg)^2 dx + 3\eps^3\norm{\pxx\uebdg(t\cdot)}^2_{L^2(\R)}\\
&\quad = \gamma \int_{\R}\uebdg^5 \Pebdg dx -3\gamma\eps^2\int_{\R}\pxx\uebdg\Pebdg dx\\
&\qquad -\frac{3\eps^2}{2}\int_{\R}\uebdg^2\px\ueb\pxx\uebdg dx -10\beta\int_{\R}\uebdg^3(\px\uebdg)^2 dx.
\end{split}
\end{equation}
Since $0<\eps <1$, due to \eqref{eq:stima-l-2-1}, \eqref{eq:p-l-infty-1}, \eqref{eq:beta-eps-2}, and the Young inequality,
\begin{equation}
\label{eq:you8}
\begin{split}
&\gamma \left \vert\int_{\R}\uebdg^5 \Pebdg dx\right\vert \le \int_{\R}\left\vert \sqrt{6}\gamma\uebdg^2\Pebdg\right\vert \left\vert\frac{\uebdg^3}{\sqrt{6}}\right\vert dx\\
&\quad \le 3\gamma^2\int_{\R}\Pebdg^2\uebdg^4 dx + \frac{1}{12}\norm{\uebdg(t,\cdot)}^6_{L^{6}(\R)}\\
&\quad \le \int_{\R}\left\vert3\sqrt{6}\gamma^2\Pebdg^2\uebdg\right\vert\left\vert\frac{\uebdg^3}{\sqrt{6}}\right\vert dx + \frac{1}{12}\norm{\uebdg(t,\cdot)}^6_{L^{6}(\R)}\\
&\quad \le 27\gamma^4\int_{\R}\Peb^4\uebdg^2 dx +  \frac{1}{6}\norm{\uebdg(t,\cdot)}^6_{L^{6}(\R)}\\
&\quad \le 27\gamma^4\norm{\Pebdg}^4_{L^{\infty}((0,T)\times\R)}\norm{\uebdg(t,\cdot)}^2_{L^{2}(\R)}+  \frac{1}{6}\norm{\uebdg(t,\cdot)}^6_{L^{6}(\R)}\\
&\quad \le 27\gamma^4C_{0}\norm{\Pebdg}^4_{L^{\infty}((0,T)\times\R)}+  \frac{1}{6}\norm{\uebdg(t,\cdot)}^6_{L^{6}(\R)}\\
&\quad \le \frac{\gamma^4}{\delta^3\gamma \eps}C(T) +  \frac{1}{6}\norm{\ueb(t,\cdot)}^6_{L^{6}(\R)} \le \frac{\gamma^3}{\delta^3\eps}C(T) +\frac{1}{6}\norm{\ueb(t,\cdot)}^6_{L^{6}(\R)}\\
&\quad \le \eps^2C(T) +\frac{1}{6}\norm{\ueb(t,\cdot)}^6_{L^{6}(\R)}\le C(T) +\frac{1}{6}\norm{\ueb(t,\cdot)}^6_{L^{6}(\R)}.
\end{split}
\end{equation}
Since $0<\eps<1$, it follows from  \eqref{eq:P-pxP-intfy-2}, \eqref{eq:Px-in-l21}, \eqref{eq:beta-eps-2} and the Young inequality that
\begin{equation}
\label{eq:you9}
\begin{split}
&-3\gamma\eps^2\int_{\R}\pxx\uebdg\Pebdg dx= 3\gamma\eps^2\int_{\R}\px\uebdg\px\Pebdg dx \\
&\quad\le 3\gamma\eps^2\left\vert \int_{\R}\px\uebdg\px\Pebdg dx\right\vert \le 3\eps^2 \int_{\R}\vert \px\uebdg\vert \left\vert\gamma\px\Pebdg\right\vert dx\\
&\quad \le \frac{3\eps^2}{2} \norm{\px\uebdg(t,\cdot)}^2_{L^{2}(\R)} +\frac{3\gamma^2\eps^2}{2} \norm{\px\Pebdg(t,\cdot)}^2_{L^{2}(\R)}\\
&\quad \le \frac{3\eps^2}{2} \norm{\px\uebdg(t,\cdot)}^2_{L^{2}(\R)} +\frac{\gamma^2\eps^2}{\delta^2\eps}C(T) \le \frac{3\eps^2}{2} \norm{\px\uebdg(t,\cdot)}^2_{L^{2}(\R)}+\eps^3C(T)\\
&\quad \le \frac{3\eps^2}{2} \norm{\px\uebdg(t,\cdot)}^2_{L^{2}(\R)}+C(T).\\
\end{split}
\end{equation}
Arguing as \cite[Lemma $2.4$]{Cd4}, we have
\begin{equation}
\label{eq:gro1}
\begin{split}
\frac{d}{dt}G_1(t) - G_1(t)+&\frac{15\eps}{4}\int_{\R}\uebdg^4(\px\ueb)^2 dx +  \frac{9\eps^3}{4}\norm{\pxx\uebdg(t,\cdot)}^2_{L^2(\R)}\\
 \le& C(T)+\eps C(T)\norm{\px\uebdg(t,\cdot)}^2_{L^2(\R)},
\end{split}
\end{equation}
where
\begin{equation}
\label{eq:def-di-G1}
G_1(t)= \frac{1}{6}\norm{\uebdg(t,\cdot)}^6_{L^6(\R)}+\frac{3\eps^2}{2}\norm{\px\uebdg(t,\cdot)}^2_{L^2(\R)}.
\end{equation}
The Gronwall Lemma, \eqref{eq:stima-l-2-1} and \eqref{eq:u0epsbeta-1} give
\begin{equation}
\label{eq:8001}
\begin{split}
G_1(t)+&\frac{15\eps}{4}e^{t} \int_{0}^{t}\!\!\!\int_{\R}e^{-s}\uebdg^4(\px\ueb)^2 dsdx + \frac{9\eps^3}{4}\int_{0}^{t}\norm{\pxx\uebdg(s,\cdot)}^2_{L^2(\R)}ds\\
\le & C_{0}e^{t} + C(T)e^{t}\int_{0}^{t} e^{-s}ds + \eps C(T)e^{t}\int_{0}^{t} e^{-s}\norm{\px\uebdg(s,\cdot)}^2_{L^2(\R)}ds\\
\le & C(T) + \eps C(T)\int_{0}^{t} \norm{\px\uebdg(s,\cdot)}^2_{L^2(\R)}ds\le C(T).
\end{split}
\end{equation}
\eqref{eq:def-di-G1} and \eqref{eq:8001} give \eqref{eq:u-in-l-6}, \eqref{eq:eps-px-u-l-21}, \eqref{eq:p001} and \eqref{eq:p002}.

Arguing as \cite[Lemma $2.4$]{Cd4}, we have \eqref{eq:000101} and \eqref{eq:defuxx}.
\end{proof}
Arguing as Lemmas \ref{lm:dist-solution} and \ref{lm:entropy-solution} we have the following results
\begin{lemma}\label{lm:dist-solution-1}
Assume that \eqref{eq:assinit2}, \eqref{eq:def-di-P02}, \eqref{eq:L-2P02}, \eqref{eq:u0epsbeta-1}, and \eqref{eq:beta-eps-2} hold. Then for any compactly supported entropy--entropy flux pair $(\eta,\, q)$, there exist four sequences $\{\eps_{k}\}_{k\in\N}$, $\{\beta_{k}\}_{k\in\N}$, $\{\delta_{k}\}_{k\in\N}$, $\{\gamma_{k}\}_{k\in\N}$,  with $\eps_k, \beta_k, \delta_k, \gamma_k \to 0$,  and a limit function
\begin{equation*}
u\in L^{\infty}(0,T; L^2(\R)\cap L^6(\R)),
\end{equation*}
 such that
\begin{equation}
\label{eq:con11}
 u_{\eps_k, \beta_k, \delta_k, \gamma_k}\to u \quad  \textrm{in} \quad  L^{p}_{loc}((0,T)\times\R),\quad \textrm{for each} \quad 1\le p <6,
\end{equation}
and $u$ is a distributional solution of \eqref{eq:Bu2}.
\end{lemma}

\begin{lemma}
\label{lm:entropy-solution-1}
Assume that \eqref{eq:assinit2}, \eqref{eq:def-di-P02}, \eqref{eq:L-2P02}, \eqref{eq:u0epsbeta-1}, and \eqref{eq:beta-eps-3} hold. Then,
\begin{equation}
\label{eq:con31}
u_{\eps_k, \beta_k, \delta_k \gamma_k}\to u \quad  \textrm{in} \quad  L^{p}_{loc}((0,T)\times\R),\quad \textrm{for each} \quad 1\le p <6,
\end{equation}
where $u$ is  the unique entropy solution of \eqref{eq:Bu2}.
\end{lemma}
\begin{proof}[Proof of Theorem \ref{th:main-4}]
$i)$ and $ii)$ follows from Lemmas \ref{lm:dist-solution-1}, and \ref{lm:entropy-solution-1}, while \eqref{eq:umedianulla} follows from \eqref{u-media-nulla} \eqref{eq:con11}, or \eqref{eq:con31}.
\end{proof}

\end{document}